\documentclass{amsart}
\usepackage{amssymb,amsmath}
\usepackage{graphicx}
\usepackage{subfigure}
\usepackage{upgreek}
\usepackage{tikz}
\usepackage{tikz-cd}
\usepackage{hyperref}
\usepackage{mathtools}
\usetikzlibrary{matrix}
\usepackage{cmap}
\usepackage[T1]{fontenc}
\usepackage[utf8]{inputenc}
\usetikzlibrary{matrix}
\usepackage[all]{xy}
\usepackage{optparams,eurosym}
\usepackage{amsthm}
\usepackage{amsmath,amssymb}
\usepackage{dbnsymb}
\newtheorem{theorem}{Theorem}[section]

\newtheorem{question}[theorem]{Question}

\newtheorem{lemma}[theorem]{Lemma}
\newtheorem{remark}[theorem]{Remark}
\newtheorem{proposition}[theorem]{Proposition}
\newtheorem{definition}[theorem]{Definition}
\newtheorem{non-theorem}{Non-Theorem}

\newtheorem{innercustomgeneric}{\customgenericname}
\providecommand{\customgenericname}{}
\newcommand{\newcustomtheorem}[2]{%
	\newenvironment{#1}[1]
	{%
		\renewcommand\customgenericname{#2}%
		\renewcommand\theinnercustomgeneric{##1}%
		\innercustomgeneric
	}
	{\endinnercustomgeneric}
}

\newcustomtheorem{customthm}{Theorem}
\newcustomtheorem{customlemma}{Lemma}

\newcommand{\s}{\mathfrak{s}}
\newcommand{\w}{\mathbf{w}}
\newcommand{\x}{\mathbf{x}}
\newcommand{\z}{\mathbf{z}}
\newcommand{\bk}{\mathbf{k}}

\numberwithin{equation}{section}

\begin{document}
	 \title {Floer lasagna modules from link Floer homology}
	 \author{Daren Chen}
	 \maketitle
	 
	 \begin{abstract}
	 	In this paper, we introduce the notion of Floer lasagna modules, which is inspired by the construction of skein lasagna module in \cite{morrison2019invariants} by Morrison, Walker and Wedrich. Here we use link Floer homology instead of Khovanov-Rozansky homology. We give a description of the Floer lasagna module for $4$-manifolds obtained by adding $2$-handles to the $4$-ball and we compute some examples.
	 \end{abstract}
	 \section{Introduction}
	 
	 In \cite{morrison2019invariants}, Morrison, Walker and Wedrich introduced the notion of `lasagna algebra', which is a higher dimensional analogue of a planar algebra. Roughly speaking, a lasagna algebra consists of a vector space $\mathcal{L}(S,L)$ for each pair $(S,L)$ of a link $L$ in a $3$-sphere $S$, and for each surface $\Sigma \subset B^4 \backslash (\cup_iB^4_i)$ with $ \partial \Sigma = \Sigma \cap \left(\partial B^4 \cup (\cup_i\partial B^4_i) \right) = L\cup (\cup_i L_i) $, a map \[F_{\Sigma}:\bigotimes_{i}\mathcal{L}(\partial B^4_i,L_i) \to \mathcal{L}(\partial B^4,L),\]
	 which satisfies some compatibility conditions to give it an operad structure.

	 Based on this, they proposed a generalization of Khovanov-Rozansky homology to links in the boundary of arbitrary oriented $4$-manifolds by taking $\mathcal{L}(S,L)$ to be the Khovanov-Rozansky homology $KhR^N(L)$. The outcome is a triply-graded object \[\mathcal{S}^N(W,L) = \bigoplus_{b\in\mathbb{Z}}\bigoplus_{i,j,\in \mathbb{Z}}S^{N}_{b,i,j}(W;L),\]
	 where $i,j$ are as the usual homological and quantum gradings of Khovanov-Rozansky homology, and $b$, called \textit{the blob grading}, is a new grading. The $b=0$ part $\mathcal{S}_{0,i,j}^N(W,L)$ consists of equivalence classes of lasagna fillings, where the equivalence relation is induced by applying some operations associated to the lasagna algebra. The resulting $\mathcal{S}_{0,i,j}^N(W,L)$ is later called \textit{skein lasagna module} in \cite{manolescu2020skein}, and it could be viewed as a higher dimensional generalization of skein modules of $3$-manifolds. The other blob degree part is obtained by applying the construction of blob complex in \cite{morrison2010blob}. We will focus on the blob degree $0$ part in this paper.

	 In the definition, what is important is the functoriality of the invariant $\mathcal{L}(S,L)$ in $S^3$, which makes sure the lasagna algebra has the structure of an operad. Khovanov-Rozansky homology is such a functor for links in $S^3$ and cobordisms in $S^3\times I$. Previously, we know the functoriality of Khovanov-Rozansky homology for links in $\mathbb{R}^3$, see for example \cite{ehrig2018functoriality}. In \cite{morrison2019invariants}, the authors proved the functoriality for links in $S^3$, which requires checking that the extra sweep-around move induces identity maps on $Kh^N$. See \cite[Theorem 1.1]{morrison2019invariants}. The construction of lasagna algebra, hence the equivalence relations on lasagna fillings, works for any such functors satisfying the functoriality condition.

	 Another natural candidate of such a functor is the link Floer homology $\widehat{\mathit{HFL}}$. In \cite{juhasz2016cobordisms}, Juh\'asz gave a definition of cobordism maps on $\widehat{HFL}$ using cobordism maps on sutured Floer homology. Later in \cite{zemke2019link}, Zemke gave a definition of cobordism maps on $\mathcal{HFL}^-(L)$, a version of the minus-favored link Floer homology. In \cite{juhasz2020contact}, it was shown that both definitions agree when we are setting all the variables $U_i$ to $0$ in $\mathcal{HFL}^-(L)$ and summing over all the $Spin^c$ structures. In this paper, we will mainly follow the treatment of the cobordism maps from \cite{zemke2019link}.

	 We denote by $\mathbb{L} = (L,\w,\z)$ a multi-based link in a $3$-manifold. The main object $\mathcal{FL}(W,\mathbb{L})$ of this paper is the counterpart of the skein lasagna module $\mathcal{S}_{0,i,j}^N(W,L)$ in the world of link Floer homology, i.e. if we apply the link Floer homology functor instead of the Khovanov-Rozansky functor to lasagna fillings. We call this object the \textit{Floer lasagna module}; see Definitions \ref{def:fillings}, \ref{def:equiv fillings}, \ref{def:module}. 
     It is a doubly-graded object, where the gradings are inherited from the Maslov grading and the collapsed Alexander grading of the link Floer homology, with some adjustments according to the filling surface. It also divides according to the relative homology class of the filling surface in $H_2(W,L)$.
In \cite{manolescu2020skein}, Manolescu and Neithalath gave expressions of $S^{N}_{0,i,j}(W;L)$ when $W$ is obtained by adding $2$-handles to $B^4$, in terms of the cabled Khovanov homology defined there. See \cite[Theorem 1.1]{manolescu2020skein}. The idea is to find a good representative in each equivalence class of lasagna fillings, such that the filling surface is in some `standard' position. In this paper, we give a similar characterization of Floer lasagna modules in terms of the \textit{cabled link Floer homology} $\widehat{cHFL}(\mathbb{L},K)$ which we will define later. By a slight abuse of notation, in the definition the cabled link Floer homology, we also refer to the link in $S^3$ induced by $L$ as $L$.
\begin{theorem}
\label{prop:cabled link homology}
\label{theo:main}
Let $W$ be the $4$-manifold obtained from attaching $2$-handles to $B^4$ along a framed link $K$. Let $\mathbb{L}=(L,\w,\z)$ be a multi-based link in the $3$-manifold $\partial W$ such that $L$ is disjoint from the surgery loci. Then we have a grading-preserving isomorphism: \[\Phi: \widehat{cHFL}(\mathbb{L},K)\to \mathcal{FL}(W,\mathbb{L}).\]
\end{theorem}
This is the main theorem of the paper. The proof is in the same spirit as that in \cite{manolescu2020skein}. The new feature is that on the filling surfaces $\Sigma$ of a Floer lasagna filling there are some dividing arcs $\mathcal{A}\subset \Sigma$ dividing $\Sigma$ into two subsurfaces $\Sigma_{\mathbf{z}}$ and $\Sigma_{\mathbf{w}}$. In search of a good representative, we not only need to make the filling surface in `standard' position, but also the dividing arcs. We call Floer lasagna fillings with the filling surface in `standard' position  \textit{almost model Floer lasagna fillings}, and we call those such that the dividing arcs are in `standard' position as well \textit{model Floer lasagna fillings}. See Definitions  \ref{def:model} and \ref{def:nearly-model}. We prove Theorem \ref{theo:main} by showing we can find a model Floer lasagna filling in each equivalence class. The ambiguity in choosing such a representative leads to various relations quotiented out in defining the cabled link Floer homology, namely the braid group action, pair-of-pants map and basepoint moving action.

	 Changing dividing arcs while keeping the filling surface induces a special kind of cobordism maps on link Floer homology, called the quasi-(de)stabilization maps. These first appeared in \cite{manolescu2010heegaard}, and were further developed in \cite{zemke2017quasistabilization}, \cite{zemke2019link}. We will follow the treatment of this topic in \cite[Section 4]{zemke2019link}.
	 
	 After proving Theorem \ref{theo:main}, we give some example calculations of the Floer lasagna module in the case when $W$ is obtained by attaching a $2$-handle along a $0$-framed unknot $U$ to $B^4$, i.e. $W = S^2\times D^2$, and a special kind of multi-based links $\mathbb{L}\subset \partial W = S^2\times S^1$, which intersects geometrically a capping disk of the unknot $U$ once. For such links, we get that the Floer lasagna module vanishes:

	 	 \label{thm:link0}
    
    \begin{theorem}
    	Suppose $W$ is obtained from $B^4$ by attaching a $2$-handle along a $0$-framed unknot $K$, and $\mathbb{L}=(L,\w,\z)$ is a multi-based link in $\partial W$ which intersects a capping disk of $K$ geometrically once. Then \[\mathcal{FL}(W,\mathbb{L}) =0.\]
    	\label{thm:link1}
    \end{theorem}

    See Theorem 1.1 \cite{sullivan2024kirbybeltscategorifiedprojectors} for a related computation in the setting of skein lasagne module defined using Khovanov-Rozansky homology. 

	 In general, it is not so easy to compute the cabled link Floer homology. A more general approach would ideally involve some cobordism maps for bordered link Floer homology.

	 One advantage of link Floer homology compared to Khovanov-Rozansky homology is that we can define cobordism maps for link cobordisms in any $4$-manifolds, not only $S^3\times I$. Therefore, for a given Floer lasagna filling, we can treat it as an input \[ \otimes_i v_i \in \otimes_i \widehat{HFL}(S^3_i,\mathbb{L}_i) \cong \widehat{HFL}(\sqcup_iS^3_i,\sqcup_i\mathbb{L}_i ),\] together with a link cobordism $(W\backslash(\sqcup_iS^3_i),\Sigma) $ from  $(\sqcup_iS^3_i,\sqcup_i\mathbb{L}_i )$ to $(\partial W, \mathbb{L})$. Therefore, one can define the evaluation map:
	 \begin{equation*}
	 	\begin{split}
	 		ev: \mathcal{FL}(W,\mathbb{L}) &\to \widehat{HFL}(\partial W,\mathbb{L}) \\
	 		[\mathcal{F}]&\to F_{\Sigma}(\otimes v_i).
	 	\end{split}
	 \end{equation*}
 
    One question is, how different is the Floer lasagna module $\mathcal{FL}(W,\mathbb{L})$ from the link Floer homology $\widehat{HFL}(\partial W,\mathbb{L})$? In other words:
    
    \begin{question}
    	What are the kernel and cokernel of the map $ev$?
    	\label{q:ev}
    \end{question}
    In the case when $W=B^4$, we have $\mathcal{FL}(W,\mathbb{L})\cong \widehat{HFL}(S^3,\mathbb{L})$. See Proposition \ref{lem:B^4}. In the case of Theorem \ref{thm:link1}, even though $\mathbb{L}$ is not null-homologous in $\partial W = S^2\times S^1$, the would-be link Floer homology should be $0$, which agrees with $\mathcal{FL}(W,\mathbb{L})$. See Remark \ref{remark:vanishing of hfl} for further discussion. In this way, we would like to view $\mathcal{FL}(W,\mathbb{L})$ as some kind of generalization of $\widehat{HFL}(\partial W,\mathbb{L})$. 
    
    Another question is the relation between the Floer lasagna module and the skein lasagna module, especially in the view of the spectral sequence from Khovanov homology to link Floer homology as in \cite{dowlin2018spectral}:
    
    \begin{question}
    	Is there any relation between $\mathcal{FL}(W,\mathbb{L})$ and $S^N_0(W,\mathbb{L})$?
    \end{question}

    \textbf{Organization of the paper.} In Section \ref{sec:lasag module}, we define the Floer lasagna module, and the gradings on it. In Section \ref{sec:equivalence}, we define the cabled link Floer homology $\widehat{cHFL}(L,\w,\z, K)$, and prove Theorem \ref{theo:main}. In Section \ref{sec:calculation}, we perform some calculations of Floer lasagna modules, and prove Theorems \ref{thm:link0}, \ref{thm:link1}.

    \vspace*{2mm}
    \textbf{Acknowledgements} This work was partially supported by NSF grant number DMS-2003488. The author wishes to thank Ciprian Manolescu for his generous help in writing up this paper, and the referee for pointing out several errors, and many helpful suggestions.

	 \section{Floer lasagna modules}
	 \subsection{Definition of Floer lasagna modules}
	 \label{sec:lasag module}
	 \begin{definition}
	 	A \textbf{multi-based link} $\mathbb{L}$ in a $3$-manifold $Y$ is a triple $\mathbb{L}=(L,\mathbf{w},\textbf{z})\in Y,$ of an oriented link $L$ with two sets of basepoints $\mathbf{w}$ and $\textbf{z}$, such that they alternate when we travel along each component of $L$, and there are at least two basepoints on each component of $L$.
	 \end{definition}
     The following definition is motivated by the similar one using Khovanov homology instead of link Floer homology. See \cite{morrison2019invariants} and \cite{manolescu2020skein}.
	 \begin{definition}
	 	\label{def:fillings}
	 	Let $W$ be a smooth oriented $4$-manifold, and $\mathbb{L}=(L,\mathbf{w},\textbf{z})$ a multi-based link in $\partial W$. A \textbf{Floer lasagna filling} $\mathcal{F} =(B_i,\mathbb{L}_i,\Sigma,\mathcal{A},v_i)$ of $(W,\mathbb{L})$ is the following data:
	 	\begin{itemize}
	 		\item a disjoint union of finitely many $4$-balls $B_i$, called the \textbf{input balls}, embedded in the interior of $W$, with a multi-based link $\mathbb{L}_i=(L_i,\mathbf{w}_i,\textbf{z}_i) \in \partial B_i$. 
	 		\item an oriented properly embedded surface $\Sigma$ in $W \backslash( \cup_i B_i)$, with $\Sigma \cap \partial W = L$ and $\Sigma \cap \partial B_i=L_i,$ whose induced orientations agree with those of $L$ and $L_i$.
	 		\item A set of embedded $1$-manifolds  $\mathcal{A}\subset \Sigma$, called \textbf{dividing arcs}, such that $\mathcal{A}$ intersects each component of $\Sigma$, and the components of $\Sigma \backslash \mathcal{A}$ are partitioned into type-$\mathbf{w}$ and type-$\textbf{z}$ sub-surfaces $\Sigma \backslash \mathcal{A} = \Sigma_{\mathbf{w}}\sqcup \Sigma_{\textbf{z}}$, with $\mathbf{w},\mathbf{w}_i\in \Sigma_{\mathbf{w}}$ and $\mathbf{z},\mathbf{z}_i\in\Sigma_{\textbf{z}}$. In diagrams, we will shade the $\Sigma_{\mathbf{w}}$ subsurfaces, while the $\Sigma_{\mathbf{z}}$ subsurfaces are left blank.  What's more, the intersection of $\mathcal{A}$ with $L$ divides $L$ into components such that each component contains exactly one basepoints of $\mathbf{w}\cup \textbf{z}$. Similar condition holds for each $L_i$.
	 		\item an input $v_i\in \widehat{HFL}(S^3, \mathbb{L}_i)$ for each $i$, where the link Floer homology is taken with $\mathbb{F}_2$ coefficients.
	 	\end{itemize}
	 \end{definition} 
 
 In the language of \cite{zemke2019link}, this is the same as a decorated link cobordism $(W\backslash \cup_i B_i,\Sigma^{\sigma})$ from $(\cup_iB_i,\cup_iL_i)$ to $(\partial W, L)$, with an input $v_i\in \widehat{HFL}(S^3,\mathbb{L}_i)$, where $\sigma$ is the trivial coloring which sends all elements in $\mathbf{w}\bigcup \cup_i\mathbf{w}_i$ to $U$ and all elements in $\textbf{z}\bigcup \cup_i\textbf{z}_i$ to $V$.

\begin{figure}[h]
	\[
	{
		\fontsize{10pt}{10pt}\selectfont
		\def\svgwidth{3.5in}
\begingroup%
  \makeatletter%
  \providecommand\color[2][]{%
    \errmessage{(Inkscape) Color is used for the text in Inkscape, but the package 'color.sty' is not loaded}%
    \renewcommand\color[2][]{}%
  }%
  \providecommand\transparent[1]{%
    \errmessage{(Inkscape) Transparency is used (non-zero) for the text in Inkscape, but the package 'transparent.sty' is not loaded}%
    \renewcommand\transparent[1]{}%
  }%
  \providecommand\rotatebox[2]{#2}%
  \newcommand*\fsize{\dimexpr\f@size pt\relax}%
  \newcommand*\lineheight[1]{\fontsize{\fsize}{#1\fsize}\selectfont}%
  \ifx\svgwidth\undefined%
    \setlength{\unitlength}{230.32780575bp}%
    \ifx\svgscale\undefined%
      \relax%
    \else%
      \setlength{\unitlength}{\unitlength * \real{\svgscale}}%
    \fi%
  \else%
    \setlength{\unitlength}{\svgwidth}%
  \fi%
  \global\let\svgwidth\undefined%
  \global\let\svgscale\undefined%
  \makeatother%
  \begin{picture}(1,0.80148057)%
    \lineheight{1}%
    \setlength\tabcolsep{0pt}%
    \put(0,0){\includegraphics[width=\unitlength,page=1]{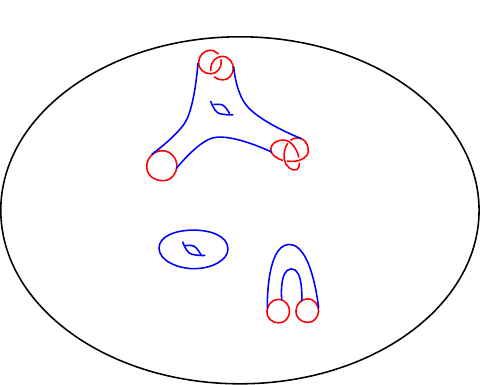}}%
    \put(0.11325282,0.24059632){\makebox(0,0)[lt]{\lineheight{1.25}\smash{\begin{tabular}[t]{l}$W^4$\end{tabular}}}}%
    \put(0.40252703,0.75940349){\makebox(0,0)[lt]{\lineheight{1.25}\smash{\begin{tabular}[t]{l}$L\subset \partial W$\end{tabular}}}}%
    \put(0.31370105,0.37658661){\makebox(0,0)[lt]{\lineheight{1.25}\smash{\begin{tabular}[t]{l}$B_1$\end{tabular}}}}%
    \put(0.58646791,0.38994984){\makebox(0,0)[lt]{\lineheight{1.25}\smash{\begin{tabular}[t]{l}$B_2$\end{tabular}}}}%
    \put(0.5880399,0.07159095){\makebox(0,0)[lt]{\lineheight{1.25}\smash{\begin{tabular}[t]{l}$B_3$\end{tabular}}}}%
    \put(0.13919319,0.53144285){\makebox(0,0)[lt]{\lineheight{1.25}\smash{\begin{tabular}[t]{l}$L_1\subset S_1$\end{tabular}}}}%
    \put(0.57820618,0.55145986){\makebox(0,0)[lt]{\lineheight{1.25}\smash{\begin{tabular}[t]{l}$L_2\subset S_2$\end{tabular}}}}%
    \put(0.70516453,0.15177014){\makebox(0,0)[lt]{\lineheight{1.25}\smash{\begin{tabular}[t]{l}$L_3\subset S_3$\end{tabular}}}}%
    \put(0,0){\includegraphics[width=\unitlength,page=2]{lasagna_filling.pdf}}%
  \end{picture}%
\endgroup%

	}
	\]
	\caption{Example of Floer lasagna modules}
	\label{fig:example-of-lasagna-modules}
\end{figure}

 We define the following equivalence relation on the set of Floer lasagna fillings.
 \begin{definition}
 	\label{def:equiv fillings}
 	Define the following relation $\sim'$ on the set of Floer lasagna fillings. For two Floer lasagna fillings $\mathcal{F}=(B_i,\mathbb{L}_i,\Sigma,\mathcal{A},v_i)$ and $\mathcal{F}'=(B'_i,\mathbb{L}'_i,\Sigma',\mathcal{A}',v'_i)$ of $(W,\mathbb{L},\mathbf{w},\textbf{z})$, we say $\mathcal{F}\sim' \mathcal{F}'$ if there exists an input ball $B'_i$ of $\mathcal{F}'$ which satisfies the following conditions:
 	 \begin{itemize}
 	 	\item $B'_i$ contains some input balls $\cup_{j\in \mathbb{J}} B_j$ of $\mathcal{F}$, and is disjoint from other input balls of $\mathcal{F}$;
 	 	\item up to isotopy relative to boundaries, $(\Sigma',\mathcal{A}') $ is the intersection of $(\Sigma,\mathcal{A})$ with $W \backslash B'_i$;
 	 	\item 
 	 	 The link $L'_i\in \partial B'_i$ is the intersection $\Sigma \cap \partial B'_i$, and there is one basepoint in each component of $L'_i\backslash\mathcal{A}$, determined by the partition $\Sigma \backslash \mathcal{A} = \Sigma_{\mathbf{w}}\sqcup \Sigma_{\textbf{z}}$;
 	 	 \item the input $v'_i$ equals $F(\otimes v_j )$, where $F$ is the cobordism map induced by the intersection of $(\Sigma,\mathcal{A})$ with $B'_i \backslash \cup_j B_j$, viewed as a decorated link cobordism from $\cup_j (B_j,\mathbb{L}_j)$ to $(B'_i,\mathbb{L}'_i)$.   	 	
 	 \end{itemize}
  
  Define the equivalence relation $\sim$ to be the transitive, symmetric and linear closure of $\sim'$ on the set of Floer lasagna fillings of $(W,\mathbb{L})$.
 \end{definition}

See \cite[Section 12]{zemke2019link} for the definition of the cobordism maps induced by decorated link cobordisms. See Figure \ref{fig:grading-calculation} for an example of equivalent Floer lasagna fillings.

\begin{figure}[h]
	\[
	{
		\fontsize{10pt}{10pt}\selectfont
		\def\svgwidth{3.5in}
\begingroup%
  \makeatletter%
  \providecommand\color[2][]{%
    \errmessage{(Inkscape) Color is used for the text in Inkscape, but the package 'color.sty' is not loaded}%
    \renewcommand\color[2][]{}%
  }%
  \providecommand\transparent[1]{%
    \errmessage{(Inkscape) Transparency is used (non-zero) for the text in Inkscape, but the package 'transparent.sty' is not loaded}%
    \renewcommand\transparent[1]{}%
  }%
  \providecommand\rotatebox[2]{#2}%
  \newcommand*\fsize{\dimexpr\f@size pt\relax}%
  \newcommand*\lineheight[1]{\fontsize{\fsize}{#1\fsize}\selectfont}%
  \ifx\svgwidth\undefined%
    \setlength{\unitlength}{218.54019757bp}%
    \ifx\svgscale\undefined%
      \relax%
    \else%
      \setlength{\unitlength}{\unitlength * \real{\svgscale}}%
    \fi%
  \else%
    \setlength{\unitlength}{\svgwidth}%
  \fi%
  \global\let\svgwidth\undefined%
  \global\let\svgscale\undefined%
  \makeatother%
  \begin{picture}(1,0.53430625)%
    \lineheight{1}%
    \setlength\tabcolsep{0pt}%
    \put(0,0){\includegraphics[width=\unitlength,page=1]{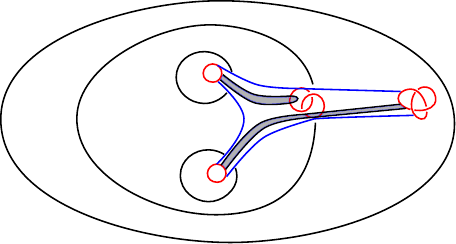}}%
    \put(0.41277547,0.32336249){\makebox(0,0)[lt]{\lineheight{1.25}\smash{\begin{tabular}[t]{l}$B_1$\end{tabular}}}}%
    \put(0.41717741,0.10534287){\makebox(0,0)[lt]{\lineheight{1.25}\smash{\begin{tabular}[t]{l}$B_2$\end{tabular}}}}%
    \put(0.18192889,0.38867159){\makebox(0,0)[lt]{\lineheight{1.25}\smash{\begin{tabular}[t]{l}$B'$\end{tabular}}}}%
    \put(0.23676086,0.18476379){\makebox(0,0)[lt]{\lineheight{1.25}\smash{\begin{tabular}[t]{l}$W'$\end{tabular}}}}%
    \put(0.40819622,0.24413838){\makebox(0,0)[lt]{\lineheight{1.25}\smash{\begin{tabular}[t]{l}$\Sigma\cap W'$\end{tabular}}}}%
    \put(0.79236388,0.22997058){\makebox(0,0)[lt]{\lineheight{1.25}\smash{\begin{tabular}[t]{l}$\Sigma'$\end{tabular}}}}%
    \put(0.69054374,0.34401556){\color[rgb]{0,0,0}\makebox(0,0)[lt]{\lineheight{1.25}\smash{\begin{tabular}[t]{l}$\mathbb{L}'$\end{tabular}}}}%
  \end{picture}%
\endgroup%

	}
	\]
	\caption{Equivalent Floer lasagna fillings}
	\label{fig:grading-calculation}
\end{figure}
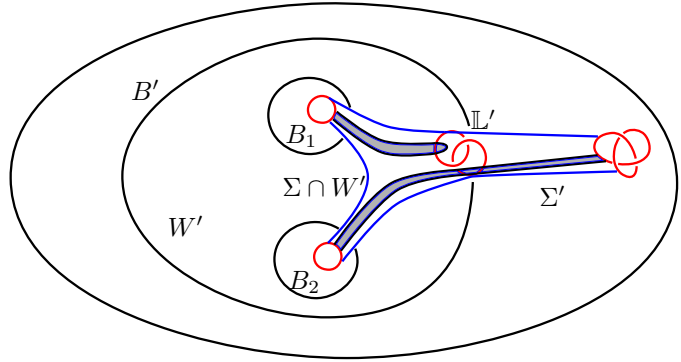
In particular, we see that if there is an isotopy of $W$ relative to $\partial W$ mapping $\mathcal{F}$ to $\mathcal{F}'$, then the two lasagna fillings $\mathcal{F},\mathcal{F}'$ are equivalent.

This equivalence relation essentially says we don't care about the `local' behaviour of the lasagna fillings, where 'local' means everything happens inside a $4$-ball in $W$.

Now we define the Floer lasagna module, which is akin to the one in \cite{morrison2019invariants} and \cite{manolescu2020skein}, except we apply the Heegaard Floer functor instead of the Khovanov-Rozansky functor.

\begin{definition}
	\label{def:module}
   For a smooth oriented $4$-manifold $W$ and a multi-based link $\mathbb{L}=(L,\mathbf{w},\mathbf{z})\in \partial W$, we define the \textbf{Floer lasagna module} $ \mathcal{FL}(W,\mathbb{L})$ as the $\mathbb{F}_2$-module generated by the equivalence classes of Floer lasagna fillings of $\mathbb{L}$ under the equivalence relation $\sim$, such that linear combinations of lasagna fillings are multilinear in the labels $v_i$.
\end{definition}

\subsection{Relative homology classes }

As in the case of skein lasagna modules in \cite{MWWHandle}, Section 2.1, we can consider the relative homology class of a Floer lasagna filling $\mathcal{F}=(B_i,\mathbb{L}_i,\Sigma,\mathcal{A},v_i)$ as follows. 

If $\left[L\right]\neq 0 \in H_1(W;\mathbb{Z})$, then there is no Floer lasagna filling of $L$, so assume $\left[L\right]= 0 \in H_1(W;\mathbb{Z})$. Let \[\partial:H_2(W,L;\mathbb{Z}) \to H_1(L;\mathbb{Z})\] be the boundary map of the long exact sequence of pairs, and denote 
\[H^L_2(W,L;\mathbb{Z}):=\partial ^{-1}([L])\subset H_2(W,L;\mathbb{Z}),\] which is torsor over $H_2(W;\mathbb{Z})$.

Each Floer lasagna filling $\mathcal{F}$ gives an element in $H^L_2(W,L;\mathbb{Z})$ by capping off the surface $\Sigma$ along each $L_i\in \partial B_i$ by some surface in each $B_i$. Equivalent Floer lasagna fillings represent the same element in $H^L_2(W,L;\mathbb{Z})$. Therefore, we have the following decomposition of the Floer lasagna module:

\[\mathcal{FL}(W,\mathbb{L}) =\bigoplus_{\alpha \in H_2^L(W,L;\mathbb{Z}) } \mathcal{FL}(W,\mathbb{L},\alpha), \]
where $\mathcal{FL}(W,\mathbb{L}, \alpha)$ is the subgroup of $\mathcal{FL}(W,\mathbb{L})$ generated by equivalence classes of Floer lasagna fillings representing the class $\alpha\in H_2^L(W,L;\mathbb{Z})$.

\subsection{Gradings}
We define a Maslov grading and an Alexander grading on the Floer lasagna module. We will use the collapsed Alexander grading instead of the multi-variable Alexander grading on $\mathbb{L}$. Otherwise more data is required in the definition of the Floer lasagna fillings, i.e. other choices of the colorings in Zemke's language. The main point in the definition is to check that it is well-defined under the equivalence relation of Floer lasagna fillings, which is a simple application of grading formulas of the link cobordism maps in \cite{zemke2019grading}.

\begin{remark}
	Note that there are two different conventions for the normalization of the absolute Maslov grading, one requires the top degree generator of $\widehat{HF}(S^3,\mathbf{w})$ to be of Maslov grading $0$, while the other requires it to be $\frac{1}{2}(|\mathbf{w}|-1)$. We use the first convention in this paper, which leads to some changes of the formulas as presented in \cite{zemke2019grading}.
\end{remark}

\begin{definition}
	For a lasagna filling $\mathcal{F} =(B_i,\mathbb{L}_i,\Sigma,\mathcal{A},v_i)$ such that each $v_i$ is a homogeneous element of $\widehat{HFL}(S^3, L_i,\mathbf{w}_i,\mathbf{z}_i)$, define the \textbf{Maslov grading} of $\mathcal{F}$ to be \[M(\mathcal{F}) = \chi(\Sigma_{\mathbf{w}}) +\sum_{i} M(v_i), \]
	and the \textbf{Alexander grading} to be \[A(\mathcal{F}) = \frac{\chi(\Sigma_{\mathbf{w}})-\chi(\Sigma_{\textbf{z}})}{2} + \sum_{i}A(v_i),\]
	where $M(v_i)$ is the absolute Maslov grading of $v_i\in \widehat{HFL}(S^3,L_i,\mathbf{w}_i,\textbf{z}_i)$ and $A(v_i)$ is the sum of the multi-Alexander grading of $v_i$. 
\end{definition}

\begin{lemma}
	The Maslov and Alexander gradings are well-defined for Floer lasagna modules.
\end{lemma}
\begin{proof}

 It is enough to show that for $\mathcal{F}=(B_i,\mathbb{L}_i,\Sigma,\mathcal{A},v_i),$ and $\mathcal{F}'=(B'_i,\mathbb{L}'_i,\Sigma',\mathcal{A}',v'_i)$ such that $\mathcal{F}\sim' \mathcal{F'}$ as in Definition \ref{def:equiv fillings}, we have
 \[M(\mathcal{F}) = M(\mathcal{F}') \text{ and } A(\mathcal{F}) = A(\mathcal{F}').\] Let $W' = B'_i \backslash \cup_{j\in \mathbb{J}} B_j$, where $B'_i$ and $B_j$ are as defined in Definition \ref{def:equiv fillings}. It follows from \cite[Theorem 1.4]{zemke2019grading}: For a cobordism $(W,\mathcal{F}):(Y_1,\mathbb{L})\to (Y_2,\mathbb{L}_2)$  \[gr_{\mathbf{w}}(F_{W,\mathfrak{s}}(\textbf{x})) - gr_{\mathbf{w}}(\textbf{x}) = \frac{c_1(\s)^2-2\chi(W)-3\sigma(W)}{4}+\chi(\Sigma_{\mathbf{w}})-\frac{|\mathbf{w}_1|+|\mathbf{w}_2|}{2},\] We apply this to the cobordism $(W', \Sigma \cap W'):(\cup_j B_j, \cup_j L_j) \to (B'_i,L'_i)$. Because of the different conventions for normalization, there is a difference with Zemke's Maslov grading $gr_{\mathbf{w}}$ and our $M$: \[gr_{\mathbf{w}}(v'_i) = M(v'_i) + \frac{1}{2}(|\mathbf{w}'_i|-1), \text{ and } gr_{\mathbf{w}}(\otimes_j v_j) = M(\otimes_j v_j) + \sum_{j\in\mathbb{J}}\frac{1}{2}(|\mathbf{w}_j|-1),\]
  As $W'$ has trivial $H_2(W';\mathbb{Z})$, it follows that \[M(v'_i)  = M(\otimes_j v_j) + \chi (\Sigma_{\mathbf{w}} \cap W')-|\mathbf{w}'_i|.\]  Since there is exactly one $\mathbf{w}$-basepoint in $\mathbf{w}'_i$ for each component of $\Sigma_{\mathbf{w}} \cap \partial B'_i,$ and $\Sigma'_{\mathbf{w}}$ is isotopic  to $\Sigma_{\mathbf{w}} \backslash (\Sigma_{\mathbf{w}}\cap W'),$ we have \[\chi(\Sigma_{\mathbf{w}}) = \chi(\Sigma_{\mathbf{w}}\cap W') + \chi(\Sigma_{\w}')-|\w'_i|,\] so \[M(\mathcal{F}') -M(\mathcal{F}) = M(v'_i) + \chi (\Sigma'_{\mathbf{w}})-M(\otimes v_j) - \chi (\Sigma_{\mathbf{w}})=0, \] as required.
 
 The verification for the well-definedness of Alexander grading follows similarly from the formula for the Alexander grading change of link cobordism maps in \cite[Theorem 1.4]{zemke2019grading} applied to the cobordism $(W', \Sigma \cap W')$, which says \[A(F(\mathbf{x})) - A(\mathbf{x}) = \frac{\chi(\Sigma_{\mathbf{w}})-\chi(\Sigma_{\mathbf{z}})}{2},\] when the cobordism $W'$ has trivial $H_2(W';\mathbb{Z})$. Note that \[\chi (\Sigma_{\mathbf{w}})-\chi(\Sigma_{\z}) = \chi(\Sigma_{\mathbf{w}}\cap W') -\chi(\Sigma_{\mathbf{z}}\cap W')  + \chi(\Sigma_{\w}')-\chi(\Sigma_{\z}'),\]
 as $|\w'_i| = |\z'_i|$, so  
 \[A(\mathcal{F}')-A(\mathcal{F}) = A(v'_i) + \frac{\chi(\Sigma'_{\mathbf{w}})-\chi(\Sigma'_{\mathbf{z}})}{2} -A(\otimes v_j)-\frac{\chi(\Sigma_{\mathbf{w}})-\chi(\Sigma_{\mathbf{z}})}{2}=0,\] as required.
\end{proof}

Let us see an example of the Floer lasagna modules, in the simplest case when $W=B^4$. From this, we can view the Floer lasagna module as an extension of the link Floer homology.
\begin{lemma}
	\label{lem:B^4}
	If $W = B^4$, and $\mathbb{L}=(L,\w,\z)$ a multi-based link in $\partial W = S^3$, then 
		\[\mathcal{FL}(B^4,\mathbb{L}) \cong \widehat{HFL}(S^3,\mathbb{L}),\]
		as doubly-graded $\mathbb{F}_2$-modules,with the multi-Alexander grading on $\widehat{HFL}$ collapsed to the single Alexander grading on $\mathcal{FL}(B^4,L,\w,\z)$ .
\end{lemma}
\begin{proof}
	For any equivalence class of Floer lasagna fillings [$\mathcal{F}]$, we can always take a particular Floer lasagna filling $\mathcal{F}' =(B',\mathbb{L}',\Sigma',\mathcal{A}',v')$ as a representative, such that $B' = W\backslash nbhd(\partial W)$ where $nbhd(\partial W)$ is a collar neighborhood of $\partial W$, and $(W\backslash B',\Sigma',\mathcal{A}')$ is the identity cobordism with $\Sigma' = L'\times I$, $\mathcal{A}' = C\times I$ for some set of points $C$ on $L$, one between each pair of adjacent basepoints $\w\cup\z$.
	This defines a map
		\begin{align*}
			\Psi:\mathcal{FL}(B^4,\mathbb{L}) &\to \widehat{HFL}(S^3,\mathbb{L}), \\
			\Psi([\mathcal{F}])&=v'.
		\end{align*}	
	It is easy to write down the inverse map $\Phi$ of $\Psi$ such that for any $v\in \widehat{HFL}(S^3,\mathbb{L})$, $\Phi(v) = \left[(B,\mathbb{L},\Sigma,\mathcal{A},v)\right]$ is the particular representative we have picked as before for each equivalence class of Floer lasagna fillings.
\end{proof}
\begin{remark}
	Note that $\Psi$ is exactly the evaluation map mentioned in Question \ref{q:ev}.
\end{remark}

	 \section{Adding 2-handles and the cabled link Floer homology}
	 
	 \label{sec:equivalence}
	 In this section, we study the effect of adding $2$-handles to the Floer lasagna modules. We will see that we can always find some good representative of an equivalence class of Floer lasagna fillings of $(W,\mathbb{L})$, such that some components of the filling surface $\Sigma$ are $L\times I$ with dividing arcs $C\times I$, and each other connected component is a $2$-disk, with a diameter as the dividing arc. Then it could be represented by an element in the so-called cabled link Floer homology 
	 $\widehat{cHFL}$, which we will define later. The treatment is very close to that in \cite{manolescu2020skein}, with some extra attention paid to the dividing arcs. The main proposition of this section is Theorem \ref{theo:main}, which we copy here:
\begin{customthm}{1.1}

	   	\label{prop:cabled link homology}
	 	Let $W$ be the $4$-manifold obtained from attaching $2$-handles to $B^4$ along a framed link $K$. Let $\mathbb{L} = (L,\w,\z)$ be a multi-based link in the $3$-manifold $\partial W$ such that $L$ is disjoint from the surgery loci. Then we have a grading-preserving isomorphism: \[\Phi: \widehat{cHFL}(\mathbb{L},K)\to \mathcal{FL}(W,\mathbb{L}).\]

	\end{customthm}

	 We begin by defining the cabled link Floer homology  $\widehat{cHFL}(\mathbb{L},K)$.
		 \subsection{Cabled link Floer homology} 
		 \label{section:cabled}
	  Let $K$ be a framed link in $S^3$. Fix a parametrization $f_i:S^1\times D^2 \to nbhd(K_i)$ for each component $K_i$ of $K$, where $S^1 = \mathbb{R}/\mathbb{Z}$. For two $n$-tuples $\mathbf{k}^{+}=\left(k_1^+,...,k_n^+\right)$ and $\mathbf{k}^{-} = \left(k_1^-,...,k_n^-\right)$, define $K(\mathbf{k}^+,\mathbf{k}^-)$ as the link obtained from $K$ by replacing each component $K_i$ with $k_i^+$ many positively oriented and $k_i^-$ many negatively oriented parallel copies along the framing of $K_i$, e.g., for each $i$, picking distinct points $x_{i,1}^-,...,x_{i,k_i^-}^-, x_{i,1}^+,...,x_{i,k_i^+}^+ \in D^2$ which lie on a fixed diameter of $D^2$, and let
	  \begin{equation}
	  	K(\mathbf{k}^+,\mathbf{k}^-) = \bigcup_if_i(S^1\times \{x_{i,1}^+,...,x_{i,k_i^+}^+, x_{i,1}^-,...,x_{i,k_i^-}^-\}).
	  	\label{eq:cables}
	  \end{equation} 
	    
	 Add two basepoints 
	 \begin{equation}
	 	\w_{i,j}^{\pm} = f_i(0,x_{i,j}^{\pm}), \quad \z_{i,j}^{\pm} = f_i(1/2,x_{i,j}^{\pm})
	 	\label{eq:basepoint}
	 \end{equation}
	 on each component of $K(\mathbf{k}^+,\mathbf{k}^-)$. Let $\w' = \cup_{i,j}\w_{i,j}$ and $\z' = \cup_{i,j}\z_{i,j}$. We denote the link Floer homology of $\left(S^3, L\cup K(\mathbf{k}^+,\mathbf{k}^-),\w\cup \w',\z\cup \z'\right)$ by $\widehat{HFL}(\mathbb{L},K,\mathbf{k}^+,\mathbf{k}^-).$ 
	 
	 There are three maps between $\widehat{HFL}(\mathbb{L},K,\mathbf{k}^+,\mathbf{k}^-)$ that we will consider: An action of some subgroup of the braid group switching parallel components of the same orientation, a pair-of-pants map combining two parallel components of opposite orientations, and a basepoint moving map on each component. We describe them more carefully.
	 
	We start with the braid group action. For each $i$, consider the subgroup $B_{k^-_{i},k^+_{i}} = \phi^{-1}(S_{k^-_{i}}\times S_{k^+_{i}}) \subset B_{k^-_{i}+k^+_{i}}$, where $\phi: B_{k^-_{i}+k^+_{i}}\to S_{k^-_{i}+k^+_{i}}$ is the natural map to the symmetric group $S_{k^-_{i}+k^+_{i}}$, and $S_{k^-_{i}}\times S_{k^+_{i}}$ is viewed as a subgroup of $S_{k^-_{i}+k^+_{i}}$. Then each element $\tau$ of $B_{k^-_{i},k^+_{i}}$ could be viewed as a map  \[ \tau: \bigcup_{j=1,...,k^-_{i}+k^+_{i}} I \to D^2\times I,\] which is an embedding of $(k^-_{i}+k^+_{i})$-many intervals into $D^2\times I$. Multiplying this embedding by $S^1$ and composing with $f_i$ gives a cobordism $f_i\circ\tau: \cup_{j} S^1\times I\to f_i(S^1\times D^2\times I) \subset S^3\times I$ from parallel copies of $K_i$ to itself. Denote the image of $f_i\circ\tau$ by $\Sigma$. Add dividing arcs $\mathcal{A}$ which are the image $f_i\circ\tau (\cup_{j}\{1/4,3/4\}\times I) $, which divides each cylinder into two halves. Call the half containing $\w'$ basepoints $\Sigma_{\w}$ and the other half $\Sigma_{\z}$. Take the union of this link cobordism with the identity cobordism over other components of $L \cup K(\mathbf{k}^+,\mathbf{k}^-)$. This cobordism gives a map \[F_{\tau}:\widehat{HFL}(\mathbb{L},K,\mathbf{k}^+,\mathbf{k}^-)\to \widehat{HFL}(\mathbb{L},K,\mathbf{k}^+,\mathbf{k}^-). \]
	 
	 This defines an action of the subgroup $B_{k^-_{i},k^+_{i}}$ for each $i$, by the functoriality of the link cobordism map as proved in \cite{zemke2019link}. See Figure \ref{fig:braid-group-action} for the corresponding cobordism of a single crossing in $B_{k^-_{i},k^+_{i}}$.
 \begin{figure}[h]
 	\[
 	{
 		\fontsize{10pt}{10pt}\selectfont
 		\def\svgwidth{2in}
\begingroup%
  \makeatletter%
  \providecommand\color[2][]{%
    \errmessage{(Inkscape) Color is used for the text in Inkscape, but the package 'color.sty' is not loaded}%
    \renewcommand\color[2][]{}%
  }%
  \providecommand\transparent[1]{%
    \errmessage{(Inkscape) Transparency is used (non-zero) for the text in Inkscape, but the package 'transparent.sty' is not loaded}%
    \renewcommand\transparent[1]{}%
  }%
  \providecommand\rotatebox[2]{#2}%
  \newcommand*\fsize{\dimexpr\f@size pt\relax}%
  \newcommand*\lineheight[1]{\fontsize{\fsize}{#1\fsize}\selectfont}%
  \ifx\svgwidth\undefined%
    \setlength{\unitlength}{191.94933463bp}%
    \ifx\svgscale\undefined%
      \relax%
    \else%
      \setlength{\unitlength}{\unitlength * \real{\svgscale}}%
    \fi%
  \else%
    \setlength{\unitlength}{\svgwidth}%
  \fi%
  \global\let\svgwidth\undefined%
  \global\let\svgscale\undefined%
  \makeatother%
  \begin{picture}(1,0.79726963)%
    \lineheight{1}%
    \setlength\tabcolsep{0pt}%
    \put(0,0){\includegraphics[width=\unitlength,page=1]{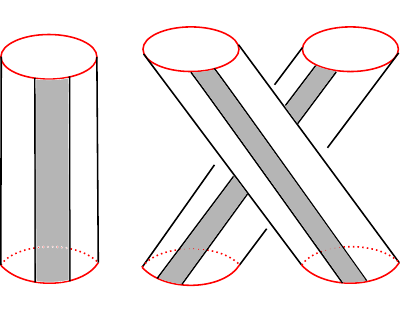}}%
    \put(0.43254158,0.7489537){\makebox(0,0)[lt]{\lineheight{1.25}\smash{\begin{tabular}[t]{l}$K$\end{tabular}}}}%
    \put(0.85986781,0.74443091){\makebox(0,0)[lt]{\lineheight{1.25}\smash{\begin{tabular}[t]{l}$K$\end{tabular}}}}%
    \put(0.09481728,0.72457372){\makebox(0,0)[lt]{\lineheight{1.25}\smash{\begin{tabular}[t]{l}$K$\end{tabular}}}}%
    \put(0.41423358,0){\makebox(0,0)[lt]{\lineheight{1.25}\smash{\begin{tabular}[t]{l}$K$\end{tabular}}}}%
    \put(0.85730663,0.00597582){\makebox(0,0)[lt]{\lineheight{1.25}\smash{\begin{tabular}[t]{l}$K$\end{tabular}}}}%
    \put(0.11103759,0.00063162){\makebox(0,0)[lt]{\lineheight{1.25}\smash{\begin{tabular}[t]{l}$K$\end{tabular}}}}%
  \end{picture}%
\endgroup%

 	}
 	\]
 	\caption{Decorated link cobordism for $\tau_{i,i+1} \in B_n$ }
 	\label{fig:braid-group-action}
 \end{figure}
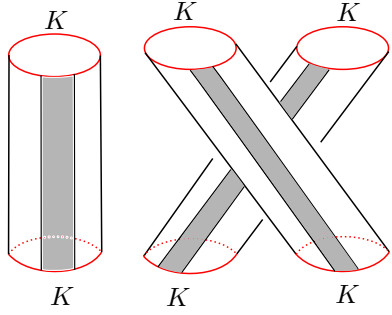

	 Now we look at the pair-of-pants map. Two oppositely oriented parallel copies of $K_i$ bounds an annulus $A$ in $S^3\times \left\{1\right\}$. Isotope $A$ into $S^3\times I$ so that it intersects with $S^3\times \{0\}$ in an unknot $U$. Remove the disk in $A$ bounded by $U$ gives a cobordism from $U$ to two oppositely oriented parallel copies of $K_i$, which is a pair-of-pants. Put a dividing arc on this pair-of-pants as in Figure \ref{fig:pair-of-pants}. Note that as we will also quotient out basepoint moving maps, the specific choice of the decoration does not matter. In particular, $\mathbb{U} = (U,\w',\z')$ has exactly two basepoints on it. Together with the identity cobordism over other components, it gives a map \[F_{P}:\widehat{HFL}(\mathbb{L}\cup \mathbb{U},K,\mathbf{k}^+,\mathbf{k}^-)\to \widehat{HFL}(\mathbb{L},K,\mathbf{k}^++e_i,\mathbf{k}^-+e_i),\]
	 where $e_i $ is the $n$-tuple with $1$ in the $i$th slot and $0$ elsewhere.
  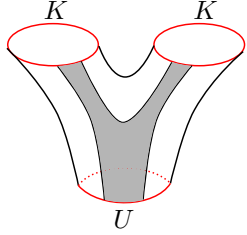
\begin{figure}[h]
 	\[
 	{
 		\fontsize{10pt}{10pt}\selectfont
 		\def\svgwidth{1.25in}
\begingroup%
  \makeatletter%
  \providecommand\color[2][]{%
    \errmessage{(Inkscape) Color is used for the text in Inkscape, but the package 'color.sty' is not loaded}%
    \renewcommand\color[2][]{}%
  }%
  \providecommand\transparent[1]{%
    \errmessage{(Inkscape) Transparency is used (non-zero) for the text in Inkscape, but the package 'transparent.sty' is not loaded}%
    \renewcommand\transparent[1]{}%
  }%
  \providecommand\rotatebox[2]{#2}%
  \newcommand*\fsize{\dimexpr\f@size pt\relax}%
  \newcommand*\lineheight[1]{\fontsize{\fsize}{#1\fsize}\selectfont}%
  \ifx\svgwidth\undefined%
    \setlength{\unitlength}{121.48470193bp}%
    \ifx\svgscale\undefined%
      \relax%
    \else%
      \setlength{\unitlength}{\unitlength * \real{\svgscale}}%
    \fi%
  \else%
    \setlength{\unitlength}{\svgwidth}%
  \fi%
  \global\let\svgwidth\undefined%
  \global\let\svgscale\undefined%
  \makeatother%
  \begin{picture}(1,0.95166191)%
    \lineheight{1}%
    \setlength\tabcolsep{0pt}%
    \put(0,0){\includegraphics[width=\unitlength,page=1]{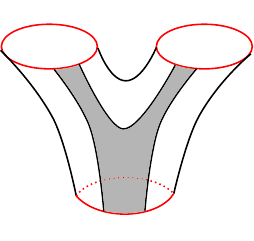}}%
    \put(0.15655535,0.87532134){\makebox(0,0)[lt]{\lineheight{1.25}\smash{\begin{tabular}[t]{l}$K$\end{tabular}}}}%
    \put(0.78385013,0.87392113){\makebox(0,0)[lt]{\lineheight{1.25}\smash{\begin{tabular}[t]{l}$K$\end{tabular}}}}%
    \put(0.4477994,0.00158936){\makebox(0,0)[lt]{\lineheight{1.25}\smash{\begin{tabular}[t]{l}$U$\end{tabular}}}}%
  \end{picture}%
\endgroup%

 	}
 	\]
 	\caption{Pair-of-pants cobordism}
 	\label{fig:pair-of-pants}
 \end{figure}

	 The pair-of-pants map is the composition of a quasi-stabilization map with a band surgery map. See \cite[Section 4, 6]{zemke2019link} for details of the definition.
	 
	 Note that \[\widehat{HFL}(\mathbb{L}\cup \mathbb{U},K,\mathbf{k}^+,\mathbf{k}^-)\cong \widehat{HFL}(\mathbb{L},K,\mathbf{k}^+,\mathbf{k}^-)\otimes V,\] where $V=\mathbb{F}_2\langle T,B\rangle$ is a $2$-dimensional $\mathbb{F}_2$-vector space generated by two elements $T,B$, with $M(T)=0, M(B)= -1$, and $A(T)=A(B)=0$.

	 The death cobordism $D$ capping the unknot $\mathbb{U}$, $D: \widehat{HFL}(\mathbb{L}\cup \mathbb{U},K,\mathbf{k}^+,\mathbf{k}^-) \to \widehat{HFL}(\mathbb{L},K,\mathbf{k}^+,\mathbf{k}^-)$ acts by 
	 \begin{equation}
	 	 D(v\otimes B)=v,\quad \quad D(v\otimes T)=0.
	 	 \label{eq:death cobordism}
	 \end{equation}
 See Figure \ref{fig:death-cobordism} for the decorated link cobordism inducing $D$.

	   \begin{figure}[h]
	 	\[
	 	{
	 		\fontsize{10pt}{10pt}\selectfont
	 		\def\svgwidth{1in}
\begingroup%
  \makeatletter%
  \providecommand\color[2][]{%
    \errmessage{(Inkscape) Color is used for the text in Inkscape, but the package 'color.sty' is not loaded}%
    \renewcommand\color[2][]{}%
  }%
  \providecommand\transparent[1]{%
    \errmessage{(Inkscape) Transparency is used (non-zero) for the text in Inkscape, but the package 'transparent.sty' is not loaded}%
    \renewcommand\transparent[1]{}%
  }%
  \providecommand\rotatebox[2]{#2}%
  \newcommand*\fsize{\dimexpr\f@size pt\relax}%
  \newcommand*\lineheight[1]{\fontsize{\fsize}{#1\fsize}\selectfont}%
  \ifx\svgwidth\undefined%
    \setlength{\unitlength}{113.15263096bp}%
    \ifx\svgscale\undefined%
      \relax%
    \else%
      \setlength{\unitlength}{\unitlength * \real{\svgscale}}%
    \fi%
  \else%
    \setlength{\unitlength}{\svgwidth}%
  \fi%
  \global\let\svgwidth\undefined%
  \global\let\svgscale\undefined%
  \makeatother%
  \begin{picture}(1,1.31201997)%
    \lineheight{1}%
    \setlength\tabcolsep{0pt}%
    \put(0,0){\includegraphics[width=\unitlength,page=1]{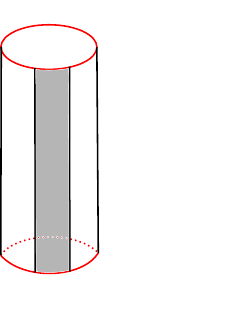}}%
    \put(0.16084562,1.230058){\makebox(0,0)[lt]{\lineheight{1.25}\smash{\begin{tabular}[t]{l}$K$\end{tabular}}}}%
    \put(0.18836135,0.00198028){\makebox(0,0)[lt]{\lineheight{1.25}\smash{\begin{tabular}[t]{l}$K$\end{tabular}}}}%
    \put(0,0){\includegraphics[width=\unitlength,page=2]{death.pdf}}%
    \put(0.74017593,0.00170639){\makebox(0,0)[lt]{\lineheight{1.25}\smash{\begin{tabular}[t]{l}$U$\end{tabular}}}}%
    \put(0,0){\includegraphics[width=\unitlength,page=3]{death.pdf}}%
  \end{picture}%
\endgroup%

	 	}
	 	\]
	 	\caption{Death cobordism}
	 	\label{fig:death-cobordism}
	 \end{figure}
	 
	 Finally there is a basepoint moving map, which is induced by the decorated link cobordism as in Figure \ref{fig:basepoint-moving} on the $j$th parallel copy of $K_i$, $j=1,...k^-_{i}+k^+_{i},$ and the identity cobordism on the rest components. We denote it by 
	 \[F_{m_{i,j}}:\widehat{HFL}(\mathbb{L},K,\mathbf{k}^+,\mathbf{k}^-)\to \widehat{HFL}(\mathbb{L},K,\mathbf{k}^+,\mathbf{k}^-).\] 
	 
	 See \cite{zemke2017quasistabilization} and \cite[Section 4]{zemke2019link} for the definition of the basepoint moving maps.

     	   \begin{figure}[h]
     	\[
     	{
     		\fontsize{10pt}{10pt}\selectfont
     		\def\svgwidth{0.5in}
\begingroup%
  \makeatletter%
  \providecommand\color[2][]{%
    \errmessage{(Inkscape) Color is used for the text in Inkscape, but the package 'color.sty' is not loaded}%
    \renewcommand\color[2][]{}%
  }%
  \providecommand\transparent[1]{%
    \errmessage{(Inkscape) Transparency is used (non-zero) for the text in Inkscape, but the package 'transparent.sty' is not loaded}%
    \renewcommand\transparent[1]{}%
  }%
  \providecommand\rotatebox[2]{#2}%
  \newcommand*\fsize{\dimexpr\f@size pt\relax}%
  \newcommand*\lineheight[1]{\fontsize{\fsize}{#1\fsize}\selectfont}%
  \ifx\svgwidth\undefined%
    \setlength{\unitlength}{47.5528747bp}%
    \ifx\svgscale\undefined%
      \relax%
    \else%
      \setlength{\unitlength}{\unitlength * \real{\svgscale}}%
    \fi%
  \else%
    \setlength{\unitlength}{\svgwidth}%
  \fi%
  \global\let\svgwidth\undefined%
  \global\let\svgscale\undefined%
  \makeatother%
  \begin{picture}(1,3.11725456)%
    \lineheight{1}%
    \setlength\tabcolsep{0pt}%
    \put(0,0){\includegraphics[width=\unitlength,page=1]{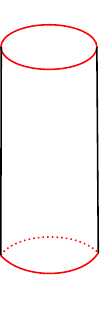}}%
    \put(0.38273403,2.92222509){\makebox(0,0)[lt]{\lineheight{1.25}\smash{\begin{tabular}[t]{l}$K$\end{tabular}}}}%
    \put(0.44820805,0){\makebox(0,0)[lt]{\lineheight{1.25}\smash{\begin{tabular}[t]{l}$K$\end{tabular}}}}%
    \put(0,0){\includegraphics[width=\unitlength,page=2]{basepoint.pdf}}%
  \end{picture}%
\endgroup%

     	}
     	\]
     	\caption{Basepoint moving cobordism}
     	\label{fig:basepoint-moving}
     \end{figure}

	 Now we give the definition of the cabled link Floer homology, which is the $\mathbb{F}_2$-vector space generated by $\widehat{HFL}(\mathbb{L},K,\mathbf{k}^+,\mathbf{k}^-)$ for all possible $n$-tuples $\mathbf{k}^+,\mathbf{k}^-$, quotienting out the relations given by maps $F_{\tau},F_P,F_{m_{i,j}}$.

	 \begin{definition}
	 	\label{def:cabled}
	 	 The \textbf{cabled link Floer homology} of $(\mathbb{L},K)$  is defined as \[\widehat{cHFL}(\mathbb{L},K) = \bigoplus_{\mathbf{k}^+,\mathbf{k}^-\in(\mathbb{Z}_{\geq 0})^n}\widehat{HFL}(\mathbb{L},K,\mathbf{k}^+,\mathbf{k}^-)\left[|\mathbf{k}^+|+|\mathbf{k}^-|\right] /\sim,\]	
 	 where $\left[l\right]$ shifts the Maslov grading up by $l$, $|\mathbf{k}^{+}|, |\mathbf{k}^{-}|$ are the sum of entries of $\mathbf{k}^{+},\mathbf{k}^{-}$ respectively, and $\sim$ is the equivalence relation which is the transitive, symmetric and linear closure of the following relations:
 	 \begin{itemize}
 	 	\item $v\sim F_{\tau}(v), \quad \forall v\in \widehat{HFL}(\mathbb{L},K,\mathbf{k}^+,\mathbf{k}^-)$;
 	 	\item $F_P(w)\sim F_D(w),\quad \forall w\in 	\widehat{HFL}(\mathbb{L}\cup \mathbb{U},K,\mathbf{k}^+,\mathbf{k}^-)$;
 	 	\item $v\sim F_{m_{i,j}}(v), \quad  \forall v\in \widehat{HFL}(\mathbb{L},K,\mathbf{k}^+,\mathbf{k}^-)$.
 	 \end{itemize}
 \end{definition} 
\begin{remark}
	Note the equivalence relation preserves both the Maslov grading and the Alexander grading, so they descend to gradings on $\widehat{cHFL}(\mathbb{L},K)$. For the grading shift of the pair-of-pants cobordism, see \cite{zemke2019grading}.
\end{remark}
\begin{remark}
	In the case when $K$ is empty, the cabled link Floer homology $\widehat{cHFL}(\mathbb{L},
	\emptyset)$ coincides with the usual link Floer homology $\widehat{\mathit{HFL}}(S^3,\mathbb{L})$ of $\mathbb{L}$ in $S^3$, by definition. In this sense, Lemma \ref{lem:B^4} is a special case of Theorem \ref{prop:cabled link homology}.
	
\end{remark}

\subsection{Almost model and model Floer lasagna fillings}

In this subsection, we give the definition of almost model and model Floer lasagna fillings, which are some particular representatives in each equivalence class of Floer lasagna fillings, which we will use later in the proof of Theorem \ref{prop:cabled link homology}. This is motivated by the proof of the similar statement as Proposition \ref{prop:cabled link homology} in the skein lasagna module in \cite{manolescu2020skein}. 

\begin{definition}
	\label{def:model}
	 Suppose $W$ is a $4$-manifold obtained from attaching $2$-handles to $B^4$ along a framed link $K$, $\mathbb{L}=(L,\w,\z)$ is a multi-based link in $\partial W$ away from the surgery loci, and $v\in \widehat{HFL}(\mathbb{L},K,\mathbf{k}^+,\mathbf{k}^-)$. A \textbf{model Floer lasagna filling} $\mathcal{F} =(B,\mathbb{L}'\cup \mathbb{K}'(\bk^+,\bk^-),\Sigma,\mathcal{A},v)$  of $(W,\mathbb{L})$ is a Floer lasagna filling such that:
	 \begin{itemize}
	 	\label {def:phi}
	 	\item 	Parameterize each $2$-handle by $f_i: D^2\times D^2\to W$, such that $K_i$ is given by $f_i(\partial D^2 \times \left\{0\right\})$. Let \[ W' = W \backslash (\bigcup _if_i(B_{\epsilon}(0)\times D^2)),\] where $B_{\epsilon}(0)$ is a small neighborhood of $0\in D^2$, i.e. we delete a small neighborhood of each $2$-handle cocore. Let $B = W'\backslash nbhd(\partial W')$, where $nbhd(\partial W')\cong [0,1) \times \partial W'$ is a collar neighborhood of $\partial W'$. By construction, $B$ is a $4$-ball.
	 	\item $L'= \{1\}\times L\subset \partial B=\left\{1\right\}\times \partial W'$ is a copy of $L$ and 
	 	\[K'(\mathbf{k}^+,\mathbf{k}^-) = \bigcup_if_i(\partial B_{\epsilon}(0)\times \{x_{i,1}^+,...,x_{i,k_i^+}^+, x_{i,1}^-,...,x_{i,k_i^-}^-\}).\]
	 	are copies of $K(\mathbf{k}^+,\mathbf{k}^-)$ which are on the boundary $\partial B_{\epsilon}(0)\times D^2 \subset \partial B$ instead of $\partial D^2\times D^2$. See Equation \ref{eq:cables}.
	 	\item $\w'$, $\z'$ consist of basepoints on $L'$ which correspond to $\w,\z$ on $L$, and   basepoints $\w_{i,j}^{\pm},\z_{i,j}^{\pm}$ on $K'(\bk^+,\bk^-)$, two on each components. See Equation \ref{eq:basepoint}. $\mathbb{L}'$ and $\mathbb{K}'(\bk^+,\bk^-)$ are the corresponding multi-based links.
	 	\item $\Sigma$ is the disjoint union of the following surface: There is a union of cylinders $L\times I\subset nbhd(\partial W')\subset W\backslash B$ connecting $L\subset \partial W$ and $L'\subset \partial B$. For each parallel copy $f_i(\partial B_{\epsilon}(0)\times x^{\pm}_{i,j})$ of $K_i$, there is a disk $f_i(B_{\epsilon(0)}\times x^{\pm}_{i,j})$ capping it, whose orientation is induced by the boundary orientation. Let \[\Sigma = (L\times I) \cup \bigcup_{i,j}f_i(B_{\epsilon(0)}\times x^{\pm}_{i,j})\] 
	 	\item The dividing arcs $\mathcal{A}$ on $L\times I$ is given by $C\times I$, where $C$ is a set of points, one in each connected component of $L\backslash(\w\cup\z)$. On the disk $f_i(B_{\epsilon(0)}\times x^{\pm}_{i,j})$, the dividing arc is given by $f_i([-\epsilon,\epsilon]\times \{0\}\times x^{\pm}_{i,j})$. We partition $\Sigma \backslash \mathcal{A} = \Sigma_{\w}\sqcup\Sigma_{\z}$ depending on whether it contains $\w'$-basepoints or $\z'$-basepoints.
	 	\item The input $v\in \widehat{HFL}(\mathbb{L},K,\mathbf{k}^+,\mathbf{k}^-)$ could be viewed as an element of $ \widehat{HFL}(L'\cup K'(\bk^+,\bk^-))$ naturally.
	 \end{itemize}
\end{definition}
     	   \begin{figure}[h]
	\[
	{
		\fontsize{10pt}{10pt}\selectfont
		\def\svgwidth{2in}
\begingroup%
  \makeatletter%
  \providecommand\color[2][]{%
    \errmessage{(Inkscape) Color is used for the text in Inkscape, but the package 'color.sty' is not loaded}%
    \renewcommand\color[2][]{}%
  }%
  \providecommand\transparent[1]{%
    \errmessage{(Inkscape) Transparency is used (non-zero) for the text in Inkscape, but the package 'transparent.sty' is not loaded}%
    \renewcommand\transparent[1]{}%
  }%
  \providecommand\rotatebox[2]{#2}%
  \newcommand*\fsize{\dimexpr\f@size pt\relax}%
  \newcommand*\lineheight[1]{\fontsize{\fsize}{#1\fsize}\selectfont}%
  \ifx\svgwidth\undefined%
    \setlength{\unitlength}{199.62495748bp}%
    \ifx\svgscale\undefined%
      \relax%
    \else%
      \setlength{\unitlength}{\unitlength * \real{\svgscale}}%
    \fi%
  \else%
    \setlength{\unitlength}{\svgwidth}%
  \fi%
  \global\let\svgwidth\undefined%
  \global\let\svgscale\undefined%
  \makeatother%
  \begin{picture}(1,1.22659225)%
    \lineheight{1}%
    \setlength\tabcolsep{0pt}%
    \put(0,0){\includegraphics[width=\unitlength,page=1]{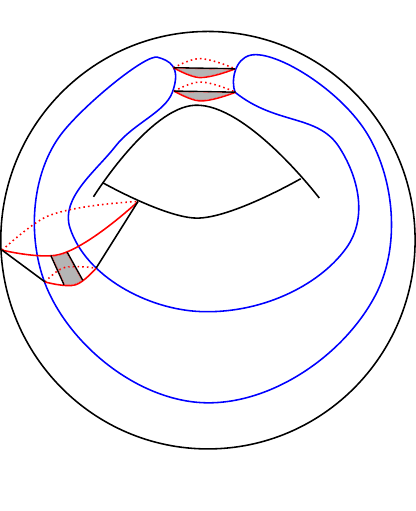}}%
    \put(0.52274532,0.34635236){\makebox(0,0)[lt]{\lineheight{1.25}\smash{\begin{tabular}[t]{l}$B$\end{tabular}}}}%
    \put(0.34807177,0.75269822){\color[rgb]{1,0,0}\makebox(0,0)[lt]{\lineheight{1.25}\smash{\begin{tabular}[t]{l}$L$\end{tabular}}}}%
    \put(0.15920421,0.47757145){\color[rgb]{1,0,0}\makebox(0,0)[lt]{\lineheight{1.25}\smash{\begin{tabular}[t]{l}$L'$\end{tabular}}}}%
    \put(0.46452084,1.17804361){\color[rgb]{1,0,0}\makebox(0,0)[lt]{\lineheight{1.25}\smash{\begin{tabular}[t]{l}$K(k_+,k_-)$\end{tabular}}}}%
    \put(0.3,0){\makebox(0,0)[lt]{\lineheight{1.25}\smash{\begin{tabular}[t]{l}$W$\end{tabular}}}}%
  \end{picture}%
\endgroup%

	}
	\]
	\caption{Diagram of a model Floer lasagna filling}
	\label{fig:cable-floer-homology}
\end{figure}
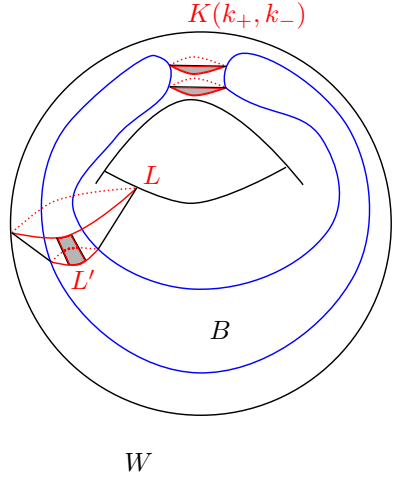

See Figure \ref{fig:cable-floer-homology} for an example of a model Floer lasagna filling.
\begin{definition}
	\label{def:nearly-model}
	If we loose the restriction on the dividing arcs on disks $f_i(B_{\epsilon}(0)\times\x_{i,j}^{\pm})$, and hence the number of basepoints on $K'(\mathbf{k}^+,\mathbf{k}^-)$ of a model Floer lasagna filling, we call such a Floer lasagna filling as an \textbf{almost model Floer lasagna filling}.
\end{definition}

In our proof of Proposition \ref{prop:cabled link homology}, we will associate a model Floer lasagna filling for each element $v\in \widehat{HFL}(\mathbb{L},K,\mathbf{k}^+,\mathbf{k}^-)$, and prove this gives a well-defined map from the cabled link Floer homology $\widehat{cHFL}(\mathbb{L},K)$ to the Floer lasagna module $\mathcal{FL}(W,\mathbb{L})$. 
To show it is a bijection, we will define the inverse map, which amounts to find a model Floer lasagna filling in each equivalence class of Floer lasagna fillings.  Roughly speaking, this is given by taking $B$ as in the definition of the model Floer lasagna filling, and use the equivalence relation on Floer lasagna fillings. The problem is that we don't have control of the dividing arcs on disks of the form $f_i(B_{\epsilon}(0)\times\x_{i,j}^{\pm})$, so we will get almost model Floer lasagna fillings, instead of model ones. In the rest of the subsection, we will present how to represent an almost model Floer lasagna filling as a linear combination of model Floer lasagna fillings, viewed as elements in $\mathcal{FL}(W,\mathbb{L})$.

The main tool is quasi-(de)stabilization map on link Floer homology. It deals with what happens to link Floer homology when we add/delete basepoints on the link.
It first appeared in \cite{manolescu2010heegaard}, and was further illustrated in \cite{zemke2017quasistabilization},\cite{zemke2019link}. The main reference of the following discussion is \cite[Section 4]{zemke2019link}.

\begin{definition}
	\label{def:stab}
	Let $\mathbb{L}=(L,\w,\z)$ be a multi-based link in $S^3$. Suppose that $w',z'$ are two new basepoints lying on the same connected component of $L \backslash (\w\cup \z)$, such that $w'$ follows $z'$ with respect to the orientation of $L$. We define the following four cobordism maps 
	\begin{align*}
		&S^+,T^+: \widehat{HFL}(L,\w,\z)\to \widehat{HFL}(L,\w\cup w',\z\cup z'), \\
		&S^-,T^-:\widehat{HFL}(L,\w\cup w',\z\cup z') \to  \widehat{HFL}(L,\w,\z),
	\end{align*}	
	induced by the following decorated link cobordisms $(S^3\times I, L\times I, \mathcal{A}_{i})$ in Figure \ref{fig:quasi-stabilization} for $i=1,2,3,4$.
	     	   \begin{figure}[h]
		\[
		{
			\fontsize{10pt}{10pt}\selectfont
			\def\svgwidth{6in}
			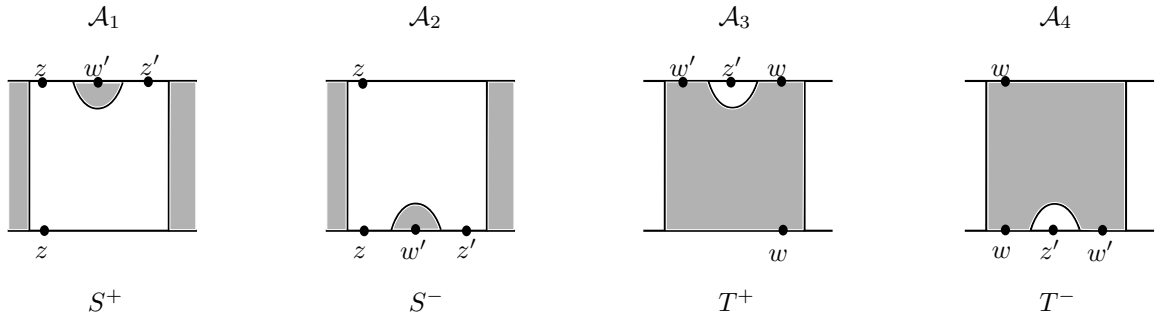
		}
		\]
		\caption{Different quasi-(de)stabilization cobordisms}
		\label{fig:quasi-stabilization}
	\end{figure}

\end{definition}
\begin{lemma}
	\label{lem: destabilization}
	Suppose $(L,\w,\z)$ is a multi-based link in $S^3$ of $m$ components, with $|\w| = |\z| =n$. Let $\w'\subset \w ,\z'\subset \z$ be subsets of the basepoints such that on each component of $L$, there is exactly one basepoint in $\w'$ and exactly one basepoint in $\z'$. Then there is an isomorphism \[\widehat{HFL}(L,\w,\z) \cong \widehat{HFL}(L,\w',\z')\otimes V^{\otimes (n-m)},\]
	where $V = \mathbb{F}_2\langle\theta^{\w},\xi^{\w} \rangle$ is a bigraded $2$-dimensional vector space over $\mathbb{F}_2$, where 
	\[M(\theta^{\w})=0,\quad \quad M(\xi^{\w})=-1,\quad \quad A(\theta^{\w}) = \dfrac{1}{2}, \quad \quad A(\xi^{\w})=-\dfrac{1}{2}.\] 
	
	Under the above isomorphism, one can write down the maps $S^{\pm},T^{\pm}$ explicitly as follows:
	\begin{equation}
		\label{eq:stab}
		\begin{split}
			&S^+(v) = v\otimes \theta^{\w}, \quad\quad S^-({v\otimes \theta^{\w}})=0,\quad\quad S^-({v\otimes \xi^{\w}})=v \\
			&T^{+}(v) = v \otimes \xi^{\w},\quad \quad T^-({v\otimes \theta^{\w}})=v,\quad\quad T^-({v\otimes \xi^{\w}})=0, 
		\end{split}
	\end{equation}
	for any $v\in \widehat{HFL}(L,\w,\z)$.
\end{lemma}	
\begin{proof}
	It follows from \cite[Proposition 5.3]{zemke2017quasistabilization} by setting all the $U$ and $V$ variables equal to $0$, and  induction on $n$. Note the difference in the Maslov grading conventions 
\end{proof}

From these descriptions, it is easy to see the following lemmas:
\begin{lemma}
	\label{lem:extension}
	Let $(L,\w,\z)$ and $(L,\w',\z')$ be as in Lemma \ref{lem: destabilization}. For any element $v\in \widehat{HFL}(L,\w',\z')$ and $u\in V^{\otimes (n-m)}$, there exists a decorated link cobordism $(S^3\times I, K\times I, \mathcal{A})$, such that 
	\[F(v) = v\otimes u,  \]
	where $F$ is the map induced by the decorated link cobordism on link Floer homology.	
\end{lemma}

\begin{proof}
	We concatenate $(S^3\times I,L\times I,\mathcal{A}_1)$ and $(S^3\times I,L\times I,\mathcal{A}_2)$, one for each pair of extra basepoints in $(\w\backslash\w')\cup (\z\backslash\z')$, depending on whether the corresponding component of $V^{\otimes (n-m)}$ is $\theta^{\w}$ or $\xi^{\w}$. The result then follows from Equation (\ref{eq:stab}).
\end{proof}

\begin{lemma}
	\label{lem:vanishing}
	If $(S^3\times I,L\times I,\mathcal{A})$ is a decorated link cobordism from $(L,\w,\z)$ to itself, such that there exists one connected component of either $\Sigma_{\mathbf{w}}$ or $\Sigma_{\mathbf{z}}$ which is a closed disk not intersecting the boundary $L\times \{0,1\}$, then the induced map $F$ on link Floer homology is $0$.
\end{lemma}

\begin{proof}
	Suppose there exists such a connected component of $\Sigma_{\mathbf{w}}$ as in the Lemma, then by shrinking this disk via isotopy, we can decompose $F$ as \[F = F_1\circ S^-\circ S^+\circ F_2,\] for some $F_1$ and $F_2$. Now \[S^-\circ S^+=0,\] so $F=0$. The situation for disks in $\Sigma_{\mathbf{z}}$ is the same by replacing $S$ with $T$.
\end{proof}

Now suppose $\mathcal{F} =(B,\mathbb{L}'\cup \mathbb{K}'(\bk^+,\bk^-),\Sigma,\mathcal{A},v)$ is a almost model Floer lasagna filling.

To simplify notation, we assume $W$ is obtained by attaching a single $2$-handle along the framed knot $K$ and $\mathbb{L}=\emptyset$, as all the operations we are going to do are concentrated near the cocore of each $2$-handle. We assume further that $\bk^+=1,\bk^-=0$, as we can do the following operation to each disk $f_i(B_{\epsilon(0)}\times x^{\pm}_{i,j})$ separately. Then $K'(\bk^+,\bk^-)$ is just a copy of the knot $K$, denoted as $K'$. Pick two basepoints $w_0,z_0$ on $K'$ among $\w',\z'$. By Lemma \ref{lem: destabilization}, we can write $v\in \widehat{HFL}(K',\w',\z')$ as  
\begin{equation}
	\label{eq:destab}
	 v= \sum_{i=1}^{k}v_i\otimes u_i,
\end{equation}
with $v_i\in \widehat{HFL}(K',w_0,z_0)$, $u_i \in V^{\otimes(|\w'|-1) }$. 

 To get a model Floer lasagna filling, we can apply some isotopy $\mathcal{I}$ of $\Sigma$ relative to the boundary, such that some dividing arc in $\mathcal{A}$ passes through the center of the disk. Then $\mathcal{F}$ is equivalent to the model Floer lasagna filling $\mathcal{F_{\mathcal{I}}}  = (B,\mathbb{K}',\Sigma',\mathcal{A}_{\mathcal{I}},F_{\mathcal{I}}(v))$, where $\mathbb{K}' = (K,w_0,z_0)$, $\Sigma'$ is a neighborhood of the center of $\Sigma$, $\mathcal{A}_{\mathcal{I}} = \mathcal{I}(\mathcal{A}) \cap \Sigma'$, and $F_{\mathcal{I}}$ is the induced map by the decorated link cobordism $(S^3\times I, \Sigma \backslash \Sigma', \mathcal{I}(\mathcal{A})\backslash \mathcal{A}_{\mathcal{I}})$.

Of course, if we apply a different isotopy $\mathcal{I'}$ instead of $\mathcal{I}$, we will get another model Floer lasagna filling $\mathcal{F_{\mathcal{I}'}}  = (B,\mathbb{K}',\Sigma',\mathcal{A}_{\mathcal{I}'},F_{\mathcal{I}'}(v))$ defined similarly. See Figure \ref{fig:isotopy} for an illustration.
Note all the data in $\mathcal{F}_{\mathcal{I}}$ and $\mathcal{F_{\mathcal{I}'}}$ are the same except $F_{\mathcal{I}}(v)$ and $F_{\mathcal{I}'}(v)$. We would like to see how different they could be.

\begin{figure}[h]
	\[
	{   \fontsize{10pt}{10pt}\selectfont
		\def\svgwidth{3in}
		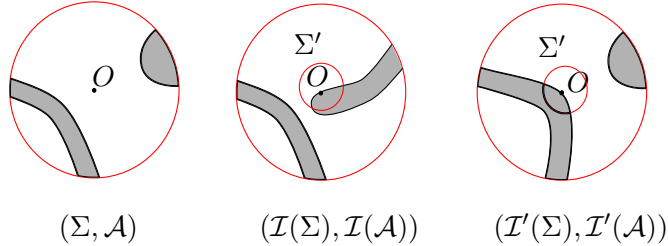
	}
	\]
	\caption{Different isotopies $\mathcal{I},\mathcal{I}'$ to $\Sigma$}
	\label{fig:isotopy}
\end{figure}

It seems a bit complicated to analyze the effect of applying isotopy to a decorated disk with some random dividing arcs on it. Instead, we first prove in Lemma \ref{lem:concatenation} that we can simplify the dividing arcs on the disk $\Sigma_i$ to the standard one as in Figure \ref{fig:standard-decoration}, using the equivalence relation on Floer lasagna fillings.  


\begin{lemma}
	\label{lem:concatenation}
	For an almost model Floer lasagna filling $\mathcal{F} =(B,(K',\w',\z'),\Sigma,\mathcal{A},v)$, we can find model Floer lasagna fillings $\mathcal{F}_i = (B, (K', w_0, z_0),\Sigma^i,\mathcal{A}_i,v_i)$ for $i=1,...,n$, such that for each $i$, the decorated disk $(\Sigma^i,\mathcal{A}_i)$ is isotopic to the following standard one in Figure \ref{fig:standard-decoration}, and \[[\mathcal{F}] = \sum_{i=1}^n[\mathcal{F}_i]\]
	in the Floer lasagna module $\mathcal{FL}(W,\emptyset)$.

	\begin{figure}[h]
		\[
		{
			\fontsize{11pt}{10pt}\selectfont
			\def\svgwidth{0.9in}
\begingroup%
  \makeatletter%
  \providecommand\color[2][]{%
    \errmessage{(Inkscape) Color is used for the text in Inkscape, but the package 'color.sty' is not loaded}%
    \renewcommand\color[2][]{}%
  }%
  \providecommand\transparent[1]{%
    \errmessage{(Inkscape) Transparency is used (non-zero) for the text in Inkscape, but the package 'transparent.sty' is not loaded}%
    \renewcommand\transparent[1]{}%
  }%
  \providecommand\rotatebox[2]{#2}%
  \newcommand*\fsize{\dimexpr\f@size pt\relax}%
  \newcommand*\lineheight[1]{\fontsize{\fsize}{#1\fsize}\selectfont}%
  \ifx\svgwidth\undefined%
    \setlength{\unitlength}{145.94688027bp}%
    \ifx\svgscale\undefined%
      \relax%
    \else%
      \setlength{\unitlength}{\unitlength * \real{\svgscale}}%
    \fi%
  \else%
    \setlength{\unitlength}{\svgwidth}%
  \fi%
  \global\let\svgwidth\undefined%
  \global\let\svgscale\undefined%
  \makeatother%
  \begin{picture}(1,0.84886737)%
    \lineheight{1}%
    \setlength\tabcolsep{0pt}%
    \put(0,0){\includegraphics[width=\unitlength,page=1]{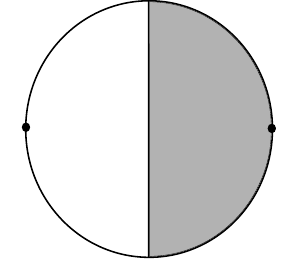}}%
    \put(0.93214457,0.39099653){\makebox(0,0)[lt]{\lineheight{1.25}\smash{\begin{tabular}[t]{l}$w_0$\end{tabular}}}}%
    \put(-0.1,0.39878302){\makebox(0,0)[lt]{\lineheight{1.25}\smash{\begin{tabular}[t]{l}$z_0$\end{tabular}}}}%
  \end{picture}%
\endgroup%

		}
		\]
		\caption{The standard decorated disk}
		\label{fig:standard-decoration}
	\end{figure}
	
\end{lemma}

\begin{proof}
	It follows from the previous lemmas. By Lemma \ref{lem:extension}, for each element $u_i\in V^{\otimes (|\w'|-1)}$ as in Equation (\ref{eq:destab}), we can find some decorated link cobordism $(S^3\times I,K\times I,\mathcal{A}'_i)$, such that \[F_i(v_i) =v_i\otimes u_i, \] where $F_i$ is the induced map on link Floer homology. Therefore, if we define $(\Sigma^i,\mathcal{A}_i)$ by concatenating $(\Sigma, \mathcal{A})$ with $(K\times I,\mathcal{A}'_i)$, and let $\mathcal{F}_i = (B,(K,w_0,z_0),\Sigma^i,\mathcal{A}_i,v_i)$, for $i =1,...,m$, we get 
	\[ [\mathcal{F}] = \sum_{i=1}^{m}[\mathcal{F}_i]  \] in the Floer lasagna module $\mathcal{FL}(W,\emptyset)$, by the equivalence relation and the linearity. 
	
	By Lemma \ref{lem:vanishing}, among all the $[\mathcal{F}_i]$'s, those that contain a closed disk as a connected component of $\Sigma^i_{\mathbf{w}}$ or $\Sigma^i_{\mathbf{z}}$ are actually $0$ in $\mathcal{FL}(W,\emptyset)$. See Figure \ref{fig:destabilization} for an illustration. Relabel the $\mathcal{F}_i$'s such that $[\mathcal{F}_i]\neq 0$ for $i=1,...,n$ and $[\mathcal{F}_i]=0$ for $i=n+1,...,m$. 
	\begin{figure}[h]
		\[
		{
			\fontsize{10pt}{10pt}\selectfont
			\def\svgwidth{3in}
			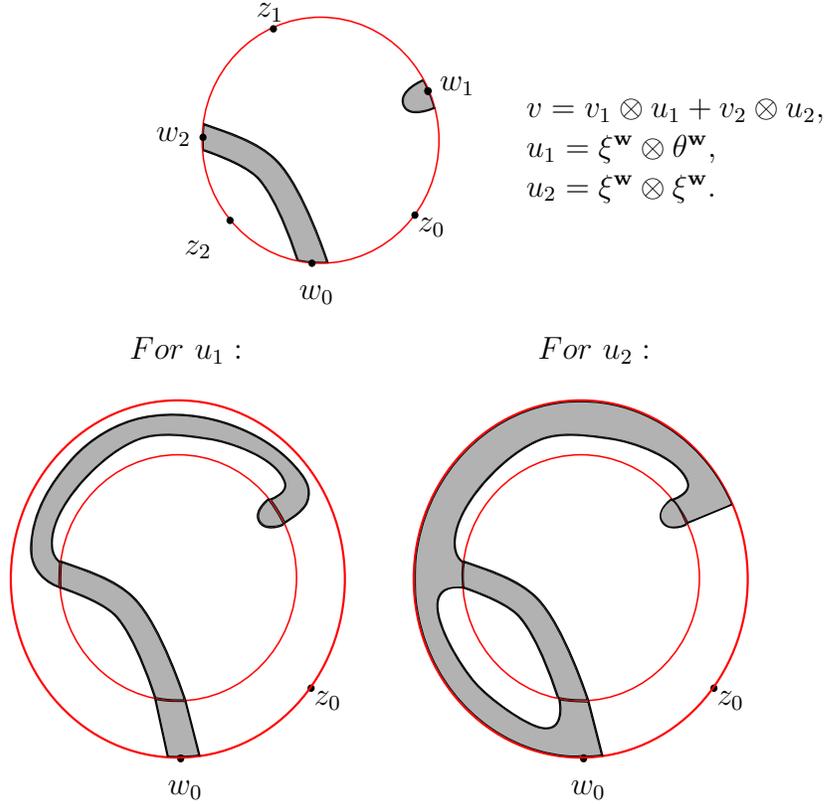
		}
		\]
		\caption{Different concatenations according to $u_i$}
		\label{fig:destabilization}
	\end{figure}
	Take an index $i\in\left\{1,...,n\right\}$. Each component of $\Sigma^i_{\mathbf{w}}$ and $\Sigma^i_{\mathbf{z}}$ intersects the boundary $\partial \Sigma^i$ non-trivially by the choice of relabeling. Since there are only two basepoints $w_0$ and $z_0$ on $\partial \Sigma^i$, and there is a basepoint in each connected component of $\Sigma^i_{\mathbf{w}}\cap \partial \Sigma^i$ and $\Sigma^i_{\mathbf{z}}\cap \partial \Sigma^i$, each of $\Sigma^i_{\mathbf{w}}$ and $\Sigma^i_{\mathbf{z}}$ has exactly one connected component, which intersects $\partial \Sigma^i$ once. If $\Sigma^i_{\w}$ is not topologically a disk, then $\Sigma^i_{\mathbf{z}} = \Sigma^i \backslash \overline{\Sigma^i_{\mathbf{w}}}$ would have more than one connected component. A similar statement holds for $\Sigma^i_{\mathbf{z}}$ as well. So both $\Sigma^i_{\mathbf{w}}$ and $\Sigma^i_{\mathbf{z}}$ are topologically a disk, and after applying some isotopy to $\Sigma^i$ relative to the boundary, it will be of the form as in Figure \ref{fig:standard-decoration}. 
	
\end{proof}

Now we can compare the differences between applying different isotopies, $\mathcal{I}$ and $\mathcal{I}'$, to an almost model Floer lasagna filling $\mathcal{F} = (B,K',\w',\z',\Sigma,\mathcal{A},v)$ of $(W ,\emptyset)$. As proved in the previous lemma, we can write \[[\mathcal{F}] = \sum_{i=1}^n [\mathcal{F}_i],\]
where each $\mathcal{F}_i = (B,(K,w_0,z_0),\Sigma^i,\mathcal{A}_i,v_i)$ is such that $(\Sigma^i,\mathcal{A}_i)$ is a standard decorated disk, as shown in Figure \ref{fig:standard-decoration}. For each $i\in \left\{1,..,n\right\}$, extend both $\mathcal{I}$ and $\mathcal{I}'$ to isotopies on $\Sigma^i$ by setting them to be identity on $K\times I$, and continue to denote these extended isotopies by $\mathcal{I}$ and $\mathcal{I}'$. Let $D^i$ be a small neighborhood of the center of $\Sigma^i$, such that $(D^i, \mathcal{I}(\mathcal{A}_i)\cap D^i)$ and $(D^i, \mathcal{I}'(\mathcal{A}_i)\cap D^i)$ are standard decorated disks. Note that each boundary of the decorated annulus $(\Sigma^i,\mathcal{I}(\mathcal{A}_i))$ (and similarly for $(\Sigma^i,\mathcal{I}'(\mathcal{A}_i))$) is the knot $K$ with only two basepoints, $w_0,z_0$. Then, the two dividing arcs $\mathcal{I}(\mathcal{A}_i)$ and $\mathcal{I}'(\mathcal{A}_i)$ on the annulus $\Sigma^i \backslash D^i$ induce two cobordism maps:
\[F_i,F'_i:\widehat{HFL}(K,w_0,z_0) \to \widehat{HFL}(K,w_0,z_0).\]

See Figure \ref{fig:different-twist} for an illustration of the situation. The claim is that $F_i$ and $F_i'$ differ by some multiple of the base-point moving maps on $\widehat{HFL}(K,w_0,z_o)$.

\begin{figure}[h]
	\[
	{
		\fontsize{10pt}{10pt}\selectfont
		\def\svgwidth{5in}
		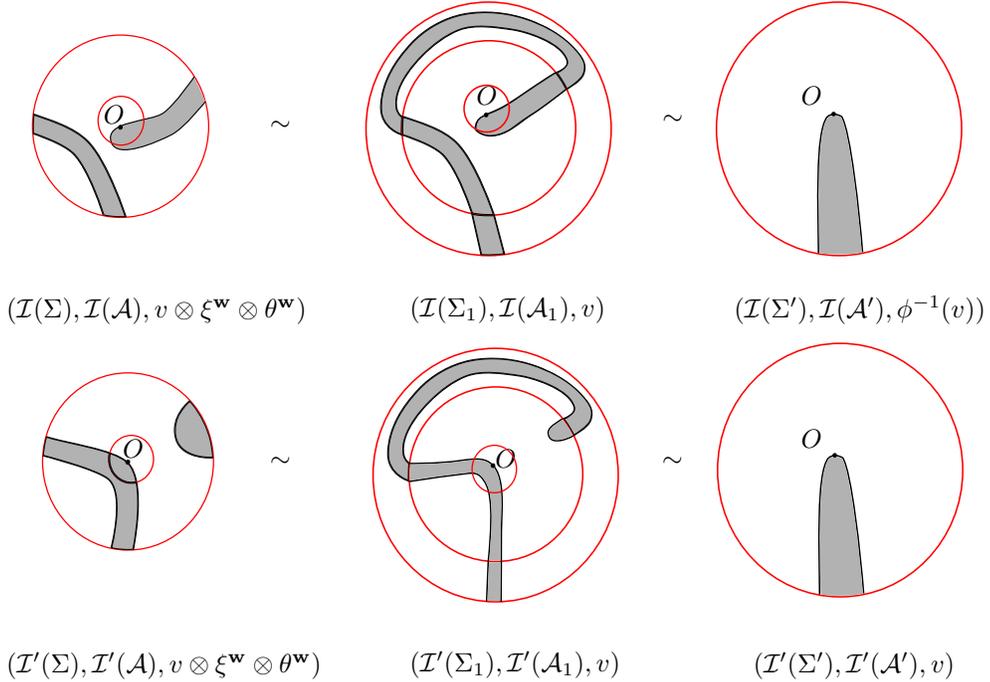
	}
	\]
	\caption{Different isotopies differing by basepoint moving}
	\label{fig:different-twist}
\end{figure}

\begin{lemma}
	\label{lem:basepoint}
	
	We have 
	\[F_i' =  \phi^{k_i}\circ F_i,\]
	for some $k_i\in \mathbb{Z}$, where $\phi$ is the basepoint moving map on $\widehat{HFL}(K,w_0,z_0)$ induced by the decorated link cobordism $(S^3\times I, K\times I, tw)$ as in Figure \ref{fig:basepoint-moving}.
\end{lemma}

\begin{proof}
	  By Lemma \ref{lem:concatenation}, the dividing arc $\mathcal{I}(\mathcal{A}_i)$ divides the filling disk $\Sigma^i$ into $\Sigma^i_{\w}$ and $\Sigma_{\mathbf{z}}^i$, both of which are disks. Then after deleting $\Sigma'$, which is a small disk centered at some point on $\mathcal{I}(\mathcal{A}_i)$, each of the subsurfaces 
	$\Sigma^i_{\w}\cap (\Sigma^i\backslash\Sigma')$ and   $\Sigma_{\mathbf{z}}^i\cap (\Sigma^i\backslash\Sigma')$ is homeomorphic to $I\times I$, with $I\times \{0\} \subset \partial \Sigma'$ and $I\times \{1\}\subset \partial \Sigma^i$. Let $[p^i_{\mathcal{I}}] \in \pi_1(\Sigma^i\backslash\Sigma', \partial (\Sigma^i\backslash\Sigma')) = \mathbb{Z}$ be the homotopy class of the path corresponding to $\{1/2\}\times I$ in $\Sigma^i_{\w}\cap (\Sigma^i\backslash\Sigma')$. Now, up to isotopy on the annulus $\Sigma^i\backslash
	\Sigma'$ relative to the boundary, such decorated link cobordism is classified by the class $ [p^i_{\mathcal{I}}]\in \mathbb{Z}$. Suppose we start with a different isotopy $\mathcal{I}'$. Then we have, up to isotopy relative to the boundary,
	\[(S^3\times I, \Sigma^i\backslash \Sigma',\mathcal{I'}(\mathcal{A})\cap (\Sigma^i\backslash\Sigma')) = (S^3\times I, K\times I,tw)^{\circ k}\circ (S^3\times I, \Sigma^i\backslash\Sigma',\mathcal{I}(\mathcal{A})\cap (\Sigma^i\backslash\Sigma')) \]
	where $\circ$ is concatenating decorated link cobordisms, and $k = [p^i_{\mathcal{I'}}] - [p^i_{\mathcal{I}}]$. Therefore, \[F'_i = \phi^{k_i}\circ F_i.\]\end{proof}

Finally, we begin the proof of Proposition \ref{prop:cabled link homology}.

\subsection{Proof of Theorem \ref{prop:cabled link homology}}
\begin{proof}[Proof of Theorem \ref{prop:cabled link homology}]
	This follows closely the treatment as in \cite[Section 4]{manolescu2020skein}, except for the dividing arcs on Floer lasagna fillings. We begin by defining the map 
	\[\Phi: \widehat{cHFL}(\mathbb{L},K)\to \mathcal{FL}(W,\mathbb{L}).\] 

	For every $ \mathbf{k}^+,\bk^-\in \mathbb{N}^n$ and every $v\in \widehat{HFL}(\mathbb{L},K,\mathbf{k}^+,\mathbf{k}^-)$, we define the Floer lasagna filling $\tilde{\Phi}(v) =(B,\mathbb{L}'\cup \mathbb{K}'(\bk^+,\bk^-),\Sigma,\mathcal{A},v)$ to be the model Floer lasagna filling as in Definition \ref{def:model}.

  Now we check that equivalent elements in $\widehat{HFL}(\mathbb{L},K,\bk^+,\bk^-)$ give equivalent Floer lasagna fillings. For the braid group action and the basepoint moving maps, it is easy to see that \[\tilde{\Phi}(v) = \tilde{\Phi}(F_{\tau}(v)), \quad \tilde{\Phi}(v) = \tilde{\Phi}(F_{m_{i,j}}(v)) \quad \text{in } \mathcal{FL}(W,\mathbb{L})\]
  by gluing the corresponding cobordism to the interior of $B$, and using functoriality of the link cobordism maps on link Floer homology
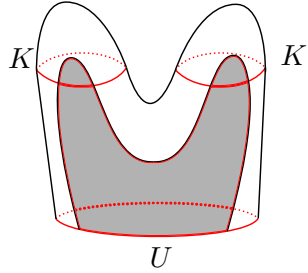
\begin{figure}[h]
	\[
	{
		\fontsize{11pt}{10pt}\selectfont
		\def\svgwidth{1.5in}
\begingroup%
  \makeatletter%
  \providecommand\color[2][]{%
    \errmessage{(Inkscape) Color is used for the text in Inkscape, but the package 'color.sty' is not loaded}%
    \renewcommand\color[2][]{}%
  }%
  \providecommand\transparent[1]{%
    \errmessage{(Inkscape) Transparency is used (non-zero) for the text in Inkscape, but the package 'transparent.sty' is not loaded}%
    \renewcommand\transparent[1]{}%
  }%
  \providecommand\rotatebox[2]{#2}%
  \newcommand*\fsize{\dimexpr\f@size pt\relax}%
  \newcommand*\lineheight[1]{\fontsize{\fsize}{#1\fsize}\selectfont}%
  \ifx\svgwidth\undefined%
    \setlength{\unitlength}{149.09824371bp}%
    \ifx\svgscale\undefined%
      \relax%
    \else%
      \setlength{\unitlength}{\unitlength * \real{\svgscale}}%
    \fi%
  \else%
    \setlength{\unitlength}{\svgwidth}%
  \fi%
  \global\let\svgwidth\undefined%
  \global\let\svgscale\undefined%
  \makeatother%
  \begin{picture}(1,0.92875033)%
    \lineheight{1}%
    \setlength\tabcolsep{0pt}%
    \put(0,0){\includegraphics[width=\unitlength,page=1]{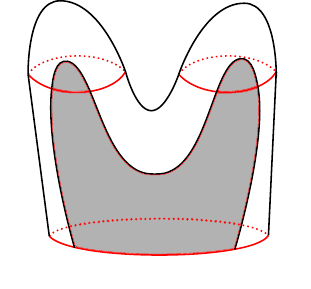}}%
    \put(0.48559361,0.00129501){\makebox(0,0)[lt]{\lineheight{1.25}\smash{\begin{tabular}[t]{l}$U$\end{tabular}}}}%
    \put(0.94164122,0.70632667){\makebox(0,0)[lt]{\lineheight{1.25}\smash{\begin{tabular}[t]{l}$K$\end{tabular}}}}%
    \put(-0.00856374,0.69616406){\makebox(0,0)[lt]{\lineheight{1.25}\smash{\begin{tabular}[t]{l}$K$\end{tabular}}}}%
  \end{picture}%
\endgroup%

	}
	\]
	\caption{Relations given by the pair-of-pants map}
	\label{fig:death-pants}
\end{figure}
  
  For the relation involving the pair-of-pants map, note that we can isotope the death cobordism over an unknot to the composition of the pair-of-pants cobordism with two capping disks over the newly added oppositely oriented parallel copies of $K_i$. This isotopy happens in $W$, and the two capping disks are parallel; copies of the core disk of the $2$-handles. See Figure \ref{fig:death-pants} for an illustration. Therefore, for every $v \in \widehat{HFL}(\mathbb{L},K,k^+,k^-)$, we can add a standard decorated disk $(D,\mathcal{A}^{D})$ representing the death cobordism. By evaluating the death cobordism, we have the following relations in $\mathcal{FL}(W,\mathbb{L})$:
  
 \[ 0 = (B,\mathbb{L}'\cup \mathbb{K}'(\bk^+,\bk^-)\sqcup \mathbb{U},\Sigma \cup D,\mathcal{A} \cup \mathcal{A}^{D},v\otimes T), \] and \[ (B,\mathbb{L}'\cup \mathbb{K}'(\bk^+,\bk^-),\Sigma,\mathcal{A},v) = (B,\mathbb{L}'\cup \mathbb{K}'(\bk^+,\bk^-)\sqcup \mathbb{U},\Sigma\cup D,\mathcal{A}\cup \mathcal{A}^{D},v\otimes B),\]
 where we use an isomorphism $\widehat{HFL} (\mathbb{K}'(\bk^+,\bk^-)\sqcup \mathbb{U}) \cong\widehat{HFL} (\mathbb{K}'(\bk^+,\bk^-))\otimes \langle B,T \rangle$ for an unknot $\mathbb{U}$ with two basepoints which is unlinked with $\mathbb{L}'\cup \mathbb{K}'(\bk^+,\bk^-).$ 
 
 Via the isotopy in Figure \ref{fig:death-pants}, we can isotope the death cobordism to the union of a pair-of-pants with two standard decorated disks parallel to the core disk. Therefore, in $\mathcal{FL}(W,\mathbb{L})$, we have the relations \[ (B,\mathbb{L}'\cup \mathbb{K}'(\bk^+,\bk^-)\sqcup \mathbb{U},\Sigma \cup D,\mathcal{A} \cup \mathcal{A}^{D},v\otimes T) = (B,\mathbb{L}'\cup \mathbb{K}'(\bk^++1,\bk^-+1),\Sigma ,\mathcal{A},F_P(v\otimes T)), 
 \]
  and \[(B,\mathbb{L}'\cup \mathbb{K}'(\bk^+,\bk^-)\sqcup \mathbb{U},\Sigma\cup D,\mathcal{A}\cup \mathcal{A}^{D},v\otimes B) = (B,\mathbb{L}'\cup \mathbb{K}'(\bk^++1,\bk^-+1),\Sigma ,\mathcal{A} , F_P(v\otimes B)). \]

 These combined give the relations involving $F_P$:
 \[\tilde{\Phi}(F_P(v\otimes T))=0,\,\,\, \text{ and } \tilde{\Phi}(F_P(v\otimes B)) = \tilde{\Phi}(v).\]

 Therefore, we see that $\tilde{\Phi}$ descends to a well-defined map \[ \Phi:\widehat{cHFL}(\mathbb{L},K)\to \mathcal{FL}(W,\mathbb{L}).\]

  It is easy to see $\Phi$ preserves both the Maslov grading and the Alexander grading, as we have done shifts in the Maslov grading appropriately in the definition of the Maslov gradings on Floer lasagna module. 
  
  To show $\Phi$ is a bijection, we define the inverse map \[\Psi: \mathcal{FL}(W,\mathbb{L}) \to \widehat{cHFL}(\mathbb{L},K)\] as follows. Given a Floer lasagna filling $\mathcal{F} = (B_i,\mathbb{L}_i,\Sigma,\mathcal{A},v_i)$, we can always take some small neighborhood $f_i(B_{\epsilon}(0)\times D^2)$ of each cocore such that the intersection of $\Sigma$ with $f_i(B_{\epsilon}(0)\times D^2)$ could be isotoped to parallel copies of disks $f_i(B_{\epsilon}(0)\times x^{\pm}_{i,j})$, which are transverse to the cocore disk $f_i(\{0\}\times D^2)$. Let $W' = W \backslash \cup_i f_i(B_{\epsilon}(0)\times D^2)$. Take a small neighborhood $nbhd(\partial W')$ of $\partial W'$ such that the components of  $\Sigma \cap nbhd(\partial W')$ which have boundary $L\times \{0,1\}$ could be isotoped to $L\times I$, and the dividing arcs on $L\times I$ is given by $C\times I$ for some set of points $C\subset L\backslash (\w\cup \z)$.  Take $B'= W'\backslash nbhd(\partial W')$, which contains all the $B_i$. Define $L',K',\w',\z'$ as in the second and third bullet point in Definition \ref{def:model}. Let $\Sigma' = \Sigma \cap (\cup_{i}f_i(B_{\epsilon}(0)\times D^2))$, $\mathcal{A}' = \mathcal{A} \cap \Sigma'$, and $v' =F'(\otimes v_i)$, where $F'$ is the map induced by the decorated link cobordism $(B\backslash (\cup_iB_i), \Sigma \backslash \Sigma', \mathcal{A} \backslash\mathcal{A}')$.  By the equivalence relation on Floer lasagna fillings, we get 
  \[[\mathcal{F}] = [\mathcal{F'}], \text{where } \mathcal{F'} = (B',\mathbb{L}'\cup \mathbb{K}'(\mathbf{k}^+,\mathbf{k}^-),\Sigma',\mathcal{A}',v')  \]
$\mathcal{F'}$ is almost of the form $\Phi(v)$ for some $v\in \widehat{cHFL}(\mathbb{L},K)$. We get an almost model Floer lasagna filling instead of a model one. To resolve this, we take the following procedures:
  
  Apply an isotopy $\mathcal{I}_{i,j}$ relative the boundary to each disk $f_i(B_{\epsilon}(0)\times x^{\pm}_{i,j})$, such that the set of dividing arcs on it passes through the point $f_i(\{0\}\times x^{\pm}_{i,j})$, i.e., intersecting with the cocore. Extend each isotopy to an isotopy $\mathcal{I}_{i,j}$ on $W$, which is identity outside some small neighborhood of $f_i(B_{\epsilon}(0)\times x^{\pm}_{i,j})$. Denote the composition of all the isotopies $\mathcal{I}_{i,j}$ by $\mathcal{I}$. Since there are finitely many such disks, we can take some smaller neighborhood $N'' = \cup_if_i(B_{\delta}(0)\times D^2)$ of the cocores, such that the set of dividing arcs on $\mathcal{I}_{i,j}(f_i(B_{\epsilon}(0)\times x^{\pm}_{i,j}))\cap N''$ is given by a single line segment. In particular, the $\Sigma_{\w}$ and $\Sigma_{\z}$ part of $\mathcal{I}_{i,j}(f_i(B_{\epsilon}(0)\times x^{\pm}_{i,j}))\cap N''$ are both half-disks, hence there are exactly $2$ basepoints on the boundary link 
  $\partial (\mathcal{I}_{i,j}(f_i(B_{\epsilon}(0)\times x^{\pm}_{i,j}))\cap N'')$, for each $i,j$. 
  
  Now replace $B'$ by a slightly larger ball $B''$, which is $W\backslash N''$ minus some collar neighborhood of its boundary. Define $L'',K''(\mathbf{k}^+,\mathbf{k}^-)$ as before. Let $\w'',\z''$ consists of basepoints on $L''$ which correspond to $\w,\z$ on $L$, and two basepoints on each component of $K''(\mathbf{k}^+,\mathbf{k}^-)$. Let $\Sigma'' =  \mathcal{I}(\Sigma')\cap N''$, $\mathcal{A}'' = \mathcal{I}(A) \cap \Sigma''$, and $v''= F''(v')$, where $F'' $ is the map induced by the decorated link cobordism $(B''\backslash B', \mathcal{I}(\Sigma')\backslash \Sigma'', \mathcal{I}(\mathcal{A}')\backslash\mathcal{A}''). $ Then $\mathcal{F''} = (B'',L''\cup K''(\mathbf{k}^+,\mathbf{k}^-),\w'',\z'',\Sigma'',\mathcal{A}'',v'')$ is a model Floer lasagna filling, hence an element of the form $\mathcal{F''} = \tilde{\Phi}(v'')$. Define $\Psi([\mathcal{F}]) = [v'']$, the equivalence class represented by $v''$ in $\widehat{cHFL}(\mathbb{L},K)$.
   
  Now we check $\Psi$ is well-defined. Firstly, equivalent Floer lasagna fillings give the same value under $\Psi$ by the functoriality of link cobordism maps, as it amounts to filling in some of the input balls $B_i$, which won't affect the intersections with the cocore disks of the $2$-handles.
  
  
  Secondly, in the process of defining $\Psi$, we have made the following choices: 
  \begin{itemize}
  	\item An isotopy which makes the part of the filling surface $\Sigma \cap f_i(B_{\epsilon}(0))$ transverse to the cocore.
  	\item An isotopy for each copy of the disk $f_i(B_{\epsilon}(0)\times x^{\pm}_{i,j})$ which moves the dividing arcs passing through the cocore.
  \end{itemize}

The effect of the second isotopies is studied in Lemma \ref{lem:concatenation} and Lemma \ref{lem:basepoint}. More explicitly, given an almost model Floer lasagna filling and two different choices of isotopies, $\mathcal{I}$ and $\mathcal{I'}$, of the filling disks that move the dividing arcs passing through the cocore, we first apply Lemma 3.10 to replace each arbitrarily decorated filling disk in the almost model Floer lasagna filling with a linear combination of standard decorated disks. We then apply the isotopies $\mathcal{I}, \mathcal{I'}$ to move the dividing arcs to intersect the cocore, and truncate the filling disks to a smaller neighborhood of the cocore to get a model Floer lasagna filling. By Lemma \ref{lem:basepoint}, the effect of different choices of the isotopies $\mathcal{I},\mathcal{I'}$ amounts to applying a multiple of the basepoint-moving map, which is quotient out in the definition of $\widehat{cHFL}(\mathbb{L},K)$. See Figure \ref{fig:different-twist} for an illustration.

\begin{figure}[h]
	\[
	{
		\fontsize{11pt}{10pt}\selectfont
		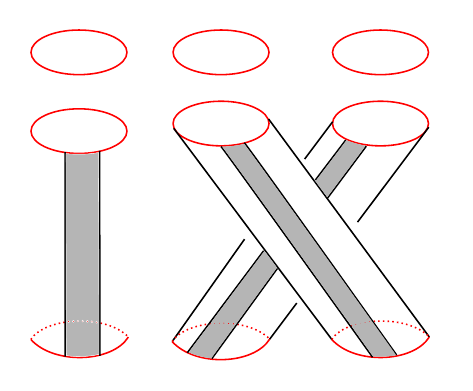
	}
	\]
	\caption{A schematic drawing of the isotopies involving braid group actions and the basepoint moving maps}
	\label{fig:braid_base_point}
\end{figure}

The ambiguity in the choice of the first isotopy leads to quotienting out the braid group action, the pair-of-pants maps and the basepoint moving maps, as explained in the proof of Theorem 1.1 in \cite[Section 4]{manolescu2020skein}. 

Here is an outline of the argument. Suppose we have chosen another isotopy of $W$ relative to $\partial W$, connect these two by a $\left[0,1\right]$-family of isotopies $\Lambda_t$. Consider the image of $\Sigma$ under this family of isotopies. 
There are finitely many times $t_k\in\left[0,1\right],\,\, k=1,...,n$ at which the image of $\Sigma$ stops being transverse to the core disk. Also, let $t_0=0$ and $t_{n+1}=1$. Choose some $\epsilon>0$ small enough, such that $t_{i}+\epsilon<t_{i+1}-\epsilon$ for each $i=0,..,n$.  Then, by adjusting the isotopy if necessary, we can assume that the intersections of $\Lambda_{t_i+\epsilon}(\Sigma)$ with the cocores agree with those of $\Lambda_{t_{i+1}-\epsilon}(\Sigma)$ with the cocores. By shrinking to a small neighborhood $N''$ of the cocores if necessary, we can assume that the decoration on each component of $ \Lambda_{t_i+\epsilon}(\Sigma) \cap N''$ is a standard decorated disk, as in Figure \ref{fig:standard-decoration}, and the same applies to each component of $\Lambda_{t_{i+1}-\epsilon}(\Sigma) \cap N''$. By composing with a braid group action cobordism as decorated in Figure \ref{fig:braid-group-action}, we can assume that the isotopy $\Lambda_t$ for $t_i+\epsilon<t<t_{i+1}-\epsilon$ fixes the surface $\Sigma $ throughout, while only moving the decorations on the surface. If there is a closed component of $\Sigma_{\mathbf{w}}$ or $\Sigma_{\mathbf{z}}$ appearing in the isotopy process, then that implies that there is such a closed component in the decorated surface $(\Sigma,\mathcal{A})$ of the lasagna filling we start with. By Lemma \ref{lem:vanishing}, that implies the lasagna filling we start with equals $0$ in $\mathcal{FL}(W,\mathbb{L})$, and we map it to $0$ in $\widehat{cHFL}$. Otherwise, by the argument as in Lemma \ref{lem:basepoint}, we get that the induced map by the isotopy on the decoration has the effect of applying some powers of the basepoint moving maps. Since we have quotiented out both the braid group action and the basepoint moving maps in the definition of $\widehat{cHFL}$, so the value of $ \Psi(\left[F\right])$ is unchanged in $\widehat{cHFL}$. See Figure \ref{fig:braid_base_point} for an illustration.


 At the finitely many times $t_i$ when $\Lambda_t(\Sigma)$ stops being transverse to the cocore disks, there will be a creation or cancellation of two intersection points of $\Lambda_t(\Sigma)$ with the cocore disks with opposite signs.  Suppose at time $t_k$ two new intersection points are created, then before time $t_k$, the map $\Phi$ gives some element  \[v\in \widehat{HFL}(L\cup U,K,\bk^+,\bk^-),\] and after time $t_k$, the map $\Phi$ gives the element \[F(v)\in\widehat{HFL}(L,K,\bk^++e_i,\bk^-+e_i),\] where $F$ is the map induced by a pair-of-pants, which is the capping disk $D$ of $U$ with standard decoration cutting out two standard decorated disks $D_1$ and $D_2$. See Figure \ref{fig:death-pants} for an illustration.




\begin{figure}[h]
	\[
	{
		\fontsize{10pt}{10pt}\selectfont
		\def\svgwidth{4in}
		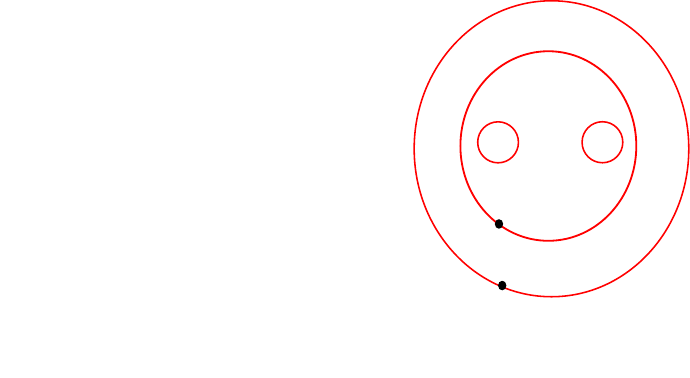
	}
	\]
	\caption{Different decorations on the pair-of-pants}
	\label{fig:different-pair-of-pants}
\end{figure}

If we allow isotopy of the remaining surface $D\backslash (D_1\cup D_2)$ which could move the boundaries, then we can isotope the decoration on $D\backslash (D_1\cup D_2)$ to one of the two decorations drawn in Figure \ref{fig:different-pair-of-pants}, where 
\[F^{\z}_B: \widehat{HFL}(U,w,w',z,z') \to \widehat{HFL}((K^+,w,z)\cup (K^-,w',z'))\] is the band surgery map from the unknot $U$ with four basepoints to the link $K^+\cup K^-$ (see \cite[Section 6]{zemke2019link} for the definition of $F^{\z}_B$), and 
\[\mathcal{F}_i:\widehat{HFL}(U,w,z) \to  \widehat{HFL}(U,w,w',z,z') \] is some quasi-stabilization map adding two basepoints $w',z'$ to the unknot, for $i=1,2$. 

These two decorations differ by the ordering of the two outputs, i.e., if we travel along the boundary of $\Sigma_{\mathbf{w}}$ clockwise, whether we go through $U,K^+,K^-$, or $U,K^-,K^+$. However, the induced map $\mathcal{F}_1$ and $\mathcal{F}_2$ on $\widehat{HFL}$ agree by the following bypass relation (see \cite[Lemma 4.7]{zemke2019link}), via the following argument, summarized in Figure \ref{fig:bypass}.

 \begin{figure}[h]
	\[
	{
		\fontsize{10pt}{10pt}\selectfont
		\def\svgwidth{4in}
		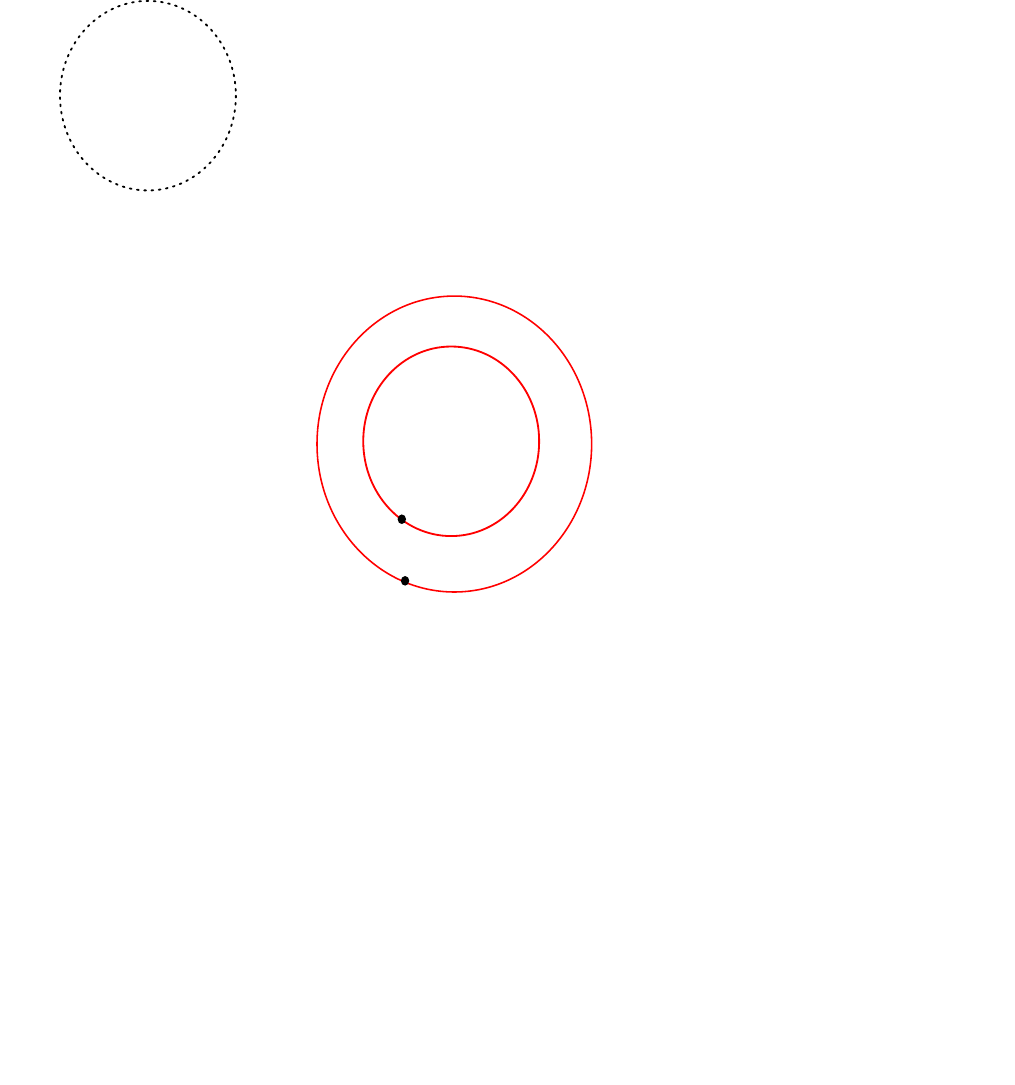
	}
	\]
	\caption{The bypass relation implies $\mathcal{F}_1=\mathcal{F}_2$}
	\label{fig:bypass}
\end{figure}
The first row of Figure \ref{fig:bypass} describes the bypass relation at the homology level. It states that if $\mathcal{F}_i,i=1,2,3$ are the induced maps on $\widehat{HFL}$ by some link cobordisms, such that they share the same underlying undecorated link cobordism, and the decorations differ only in a local disk as depicted, then they add up to $0$ at the homology level.

The second row identifies a local decorated disk in the current situation such that the first two terms of the bypass relation are exactly the quasi-stabilization maps $\mathcal{F}_1,\mathcal{F}_2$ that we are interested in. Since we work over the base field $\mathbb{F}_2$, to conclude $\mathcal{F}_1=\mathcal{F}_2$, it is enough to prove that $\mathcal{F}_3=0$.

The third row decomposes $\mathcal{F}_3$ to $S^+\circ \mathcal{G}$, where $S^+$ is the quasi-stabilization map recalled in Definition \ref{def:stab}, and $\mathcal{G}$ is the link cobordism map from $\widehat{HFL}(U,w,z)$ to itself with the given decoration. The underlying undecorated link cobordism is a cylinder $U\times [0,1]$ inside $S^3\times [0,1]$. Hence, by the grading shift formula \cite[Theorem 1.4]{zemke2019grading}, the Alexander grading shift of $\mathcal{G}$ is given by 
\[\frac{\langle c_1(\mathfrak{s}),\hat{\Sigma}\rangle- [\hat{\Sigma}]\cdot [\hat{\Sigma}]}{2}+\frac{\chi(\Sigma_{\w})-\chi(\Sigma_{\z})}{2} = 0+\frac{0-2}{2}=-1.\]

However, $\widehat{HFL}(U,w,z)$ has a single generator with the Alexander grading $0$. Hence,
\[\mathcal{G}=0, \text{ and } \mathcal{F}_3 = S^+\circ \mathcal{G}=0.\]

Therefore, these two decorated link cobordism in Figure \ref{fig:different-pair-of-pants}  give the same cobordism map at the homology level.

Isotopies which move the boundary correspond to basepoint moving maps.  Hence, by composing basepoint moving cobordisms on each boundary component if necessary, we can assume the map induced by the pair-of-pants cobordism is the $F_P$ used in the definition of the cabled link Floer homology.

Hence $\Psi$ is well-defined. From the definition of $\Psi$ and $\Phi$, it is easy to see that $\Phi$ and $\Psi$ are inverse to each other. Hence $\Phi$ gives a bijection between $\widehat{cHFL}(\mathbb{L},K)$ and  $\mathcal{FL}(W,\mathbb{L})$.
\end{proof}

	 \section{Example calculations for $0$-surgery on unknot}
	 \label{sec:calculation}
	 In this section, we give some example calculations the Floer lasagna module, using the description in terms of the cable Floer homology. 
	 
	 Suppose $K$ is the $0$-framed unknot in $S^3$, and let $W$ be the $4$-manifold obtained by attaching $D^2\times D^2$ to $B^4$ along $K$, i.e., $W=S^2\times D^2$.  We will compute  $\mathcal{FL}(W,\mathbb{L})$ when $\mathbb{L} = (L,\w,\z)$ is a link which geometrically intersects a capping disk of $K$ in $S^3$ once. We will see that  $\mathcal{FL}(W,\mathbb{L})$ turns out to be $0$, due to some grading constraint on the pair-of-pants map.

	 \begin{proposition}
	 	 Suppose $W$ is obtained from $B^4$ by attaching a $2$-handle along a $0$-framed unknot $K$, and $\mathbb{L}=(L,\w,\z)$ is a multi-based link in $\partial W$ which intersects a capping disk of $K$ geometrically once. Then \[\mathcal{FL}(W,\mathbb{L}) =0.\]
	 	 
	 	 \label{prop:linking 1}
	 \end{proposition}

	  We will use the  following result on the effect of connected sum of link cobordism maps by Zemke:
	 
	 \begin{proposition} \cite[Proposition 5.2]{zemke2019connected}
	 	Suppose that $(W_1,\Sigma_1)$ and $(W_2,\Sigma_2)$ are two link cobordisms, with chosen points $y_1\in \mathcal{A}_1\subset W_1$ and $y_2\in \mathcal{A}_2\subset W_2$, as well as an embedding of $(S^0\times D^4,N_0)$, centered at $\{y_1,y_2\}$, which is orientation preserving and maps type-$\w$ regions to type-$\w$ regions and maps type-$\z$ regions to type-$\z$ regions. Then \[F_{W_1\#W_2,\Sigma_1\#\Sigma_2,\mathfrak{s}_1\#\mathfrak{s}_2} \simeq F_{W_1\sqcup W_2,\Sigma_1\sqcup\Sigma_2,\mathfrak{s}_1\sqcup\mathfrak{s}_2},\]
	 	where the connected sum is taken at the $4$-balls centered at $y_1,y_2$.

	 	\label{prop:connected sum}
	 \end{proposition}

 In the statement, $N_0 = (S_0, \mathcal{A}_0)$ is the surface with dividing arcs defined by 
 \[S_0 = S^{0}\times \left\{(x,y,0,0):x^2+y^2\leq 1\right\} \subseteq S^0\times D^4, \,\, \mathcal{A}_0 = S^0\times \left\{(0,y,0,0):-1\leq t\leq 1\right\}.\]  The homotopy equivalence between $F_{W_1\#W_2,\Sigma_1\#\Sigma_2,\mathfrak{s}_1\#\mathfrak{s}_2}$ and $F_{W_1\sqcup W_2,\Sigma_1\sqcup\Sigma_2,\mathfrak{s}_1\sqcup\mathfrak{s}_2}$ is understood under the Kunneth formula
 \[\widehat{HFL}(Y_1\#Y_2,K_1\#K_2) \simeq \widehat{HFL}(Y_1 \sqcup Y_2, K_1 \sqcup K_2)\] applied to the targets of the two cobordisms. See Figure $5.2$ in \cite{zemke2019connected} and the discussion preceding it.
	 \begin{proof} [Proof of Proposition \ref{prop:linking 1}]

	 Suppose $\mathbb{U}=(U,\w_2,\z_2)$ is another unknot with 2 basepoints not linking with $L$ and $K$, and we want to study the map \[F_{P}:\widehat{HFL}(\mathbb{L}\cup \mathbb{U},K,k^+,k^-)\to \widehat{HFL}(\mathbb{L},K,k^++1,k^-+1),\] induced by the pair-of-pants cobordism. See Section \ref{section:cabled} for the definitions of $\widehat{HFL}(\mathbb{L},K,\mathbf{k}^+,\mathbf{k}^-)$ and $F_P$.
	 
	Because of the condition of the geometric intersection number being $1$, we can write the link 
	\newline $L\cup K(k^+,k^-)\cup U$ as the connected sum $L_1 \#(L\cup K(k^+,k^-))$, and the link $ L\cup K(k^++1,k^-+1)$ as the connected sum $L_2\#(L\cup K(k^+,k^-))$, where $L_1$, $L_2$ are the following $2$ specific links in Figure \ref{fig:trivial-links}.

\begin{figure}[h]
	\[
	{
		\fontsize{10pt}{10pt}\selectfont
		\def\svgwidth{3.5in}
		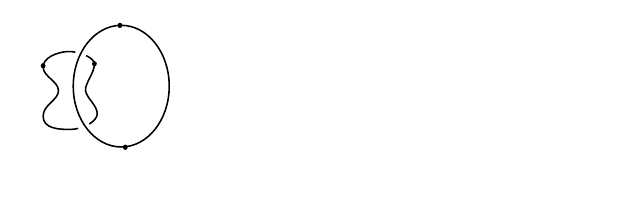
	}
	\]
	\caption{Links $L_1, L'_1$ and $L_2$}
	\label{fig:trivial-links}
	\end{figure}
	
	By the connected sum formula of link Floer homology, by choosing appropriate Heegaard diagrams and neck-stretching, we have the following isomorphisms.	
	\begin{equation*}
		 	\begin{split}
		 	\widehat{HFL}(L\cup K(k^+,k^-)\cup U)\cong \widehat{HFL}(L_1)\otimes \widehat{HFL}(L\cup K(k^+,k^-)),\\ \widehat{HFL}(L\cup K(k^++1,k^-+1))\cong \widehat{HFL}(L_2)\otimes \widehat{HFL}(L\cup K(k^+,k^-)).
		 \end{split} 
	\end{equation*}
Furthermore, the pair-of-pants cobordism from $L\cup K(k^+,k^-)\cup U$ to $ L\cup K(k^++1,k^-+1)$ could be written as a connected sum of a pair-of-pants cobordism from $L_1$ to $L_2$, with an identity cobordism from  $ \widehat{HFL}(L\cup K(k^+,k^-))$ to itself.
 Therefore, by Proposition \ref{prop:connected sum}, with respect to this isomorphism,  we have\[F_P \cong F'_P\otimes id, \] where $F'_P$ is the map induced by the pair-of-pants cobordism from $L_1$ to $L_2$.  The pair-of-pants relation in the definition of the cabled link Floer homology (See Definition \ref{def:cabled}) can be written in the form 
 \begin{equation}
 	 v\sim F'_P(B)\otimes v, \quad 0\sim F'_P(T)\otimes v,
 	 \label{eq:vanishing}
 \end{equation}
 for any $v \in \widehat{HFL}(L\cup K(k^+,k^-))$.
 
 It is enough on compute $F'_P$, which could be written as the composition of a quasi-stabilization map $T^+$ with a band surgery map $F^{\z}_B$. See Figure \ref{fig:stab-and-band} for an illustration.
\begin{figure}[h]
	\[
	{
		\fontsize{10pt}{10pt}\selectfont
		\def\svgwidth{3.5in}
		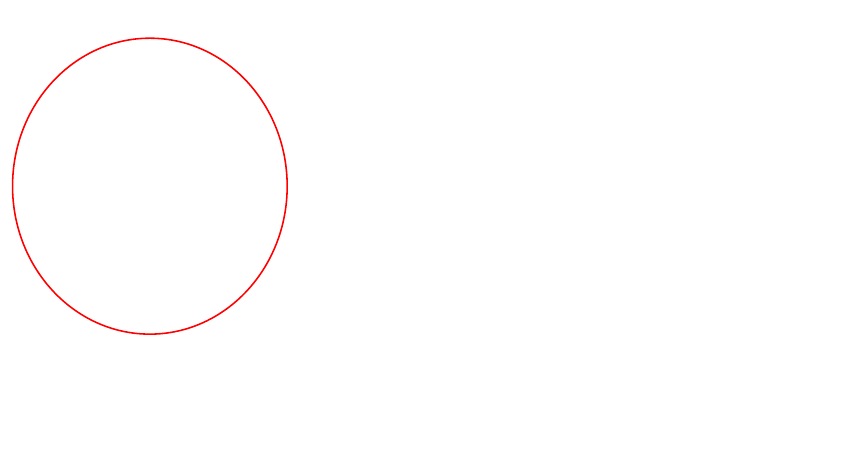
	}
	\]
	\caption{Decorated link cobordisms of $T^+$ and $F^{\z}_B$}
	\label{fig:stab-and-band}
\end{figure}

The quasi-stabilization map $T^+$ is mentioned in Definition \ref{def:stab}. See \cite[Section 6]{zemke2019link} for the definition of the band surgery map $F^{\z}_B$.

For the following discussion, we are going to use the multi-Alexander grading $(A_1,A_2)$ on links, such that $A_1$ is induced by the basepoint $w_1,z_1$, and $A_2$ is induced by the basepoint $w_2,w_3,z_2,z_3$. 
The link $L_1$ is an unlink with $2$ components, which has link Floer homology
\begin{equation*}
	\begin{split}
		 \widehat{HFL}(L_1)\cong \mathbb{F}^2_2&\cong 
		 \langle T,B\rangle, \\ \text{with }\,\, M(T)=0, \,\,M(B)=-1,  \quad &A_i(T)=A_i(B)=0, \quad \text{for }\,\, i=1,2.
	\end{split}
\end{equation*}

The link $L_1'$ is obtained from $L_1$ by adding two basepoints on $U$, so the link Floer homology of $L_1'$ is 
\begin{equation*}
	\begin{split}
		\widehat{HFL}(L'_1)\cong \mathbb{F}^4_2\cong 
		\langle T&,B\rangle \otimes \langle \theta^{\w},\xi^{\w}\rangle,\\ \text{with} \,\, M(\theta^{\w})=0,\, M(\xi^{\w})=-1, \quad A_1(\theta^{\w}) = &A_1(\xi^{\w})=0,  \quad A_2(\theta^{\w})=\dfrac{1}{2}, \,\,A_2(\xi^{\w})=-\dfrac{1}{2}.
	\end{split}
\end{equation*}
The quasi-stabilization map $T^+$ is such that \[T^+(T)  = T \otimes \xi^{\w},\quad T^+(B) = B \otimes \xi^{\w}.\]

For the link $L_2$, it is easy to get its link Floer homology, e.g. noting that it is an alternating link, and using the multi-variable Alexander polynomial \[\Delta(L_2)=x_1^{1/2}-x_1^{-1/2},\] where $x_i$ is the variable associated to link component with basepoints $w_i$ and $z_i$. Then, by \cite[Theorem 1.3]{ozsvath2008holomorphic},
\[ \widehat{HFL}(L_2,\w,\z,h) = \mathbb{F}_2^{\oplus |a_h|}\left[|h|-1\right]\]
where \[(x_1^{1/2}-x_1^{-1/2})\prod_{i=1}^3(x_i^{1/2}-x_i^{-1/2}) = \sum_{h=(h_1,h_2,h_3)\in \mathcal{H} }a_hx_1^{h_1}x_2^{h_2}x_3^{h_3}, \]
$\mathcal{H} = \left\{-1/2,1/2\right\}\times \left\{-1/2,1/2\right\} \times \left\{-1,0,1\right\},$ $|h|=h_1+h_2+h_3,$ and  $\mathbb{F}_2^{\oplus |a_h|}\left[|h|-1\right]$ means $|a_h|$ copies of $\mathbb{F}_2$ with Maslov grading $|h|-1$. In particular, $\widehat{HFL}(L_2,\w,\z)$ is a $16$-dimensional vector space over $\mathbb{F}_2$.

\begin{figure}[h]
	\[
	{
		\fontsize{10pt}{10pt}\selectfont
		\def\svgwidth{4.5in}
		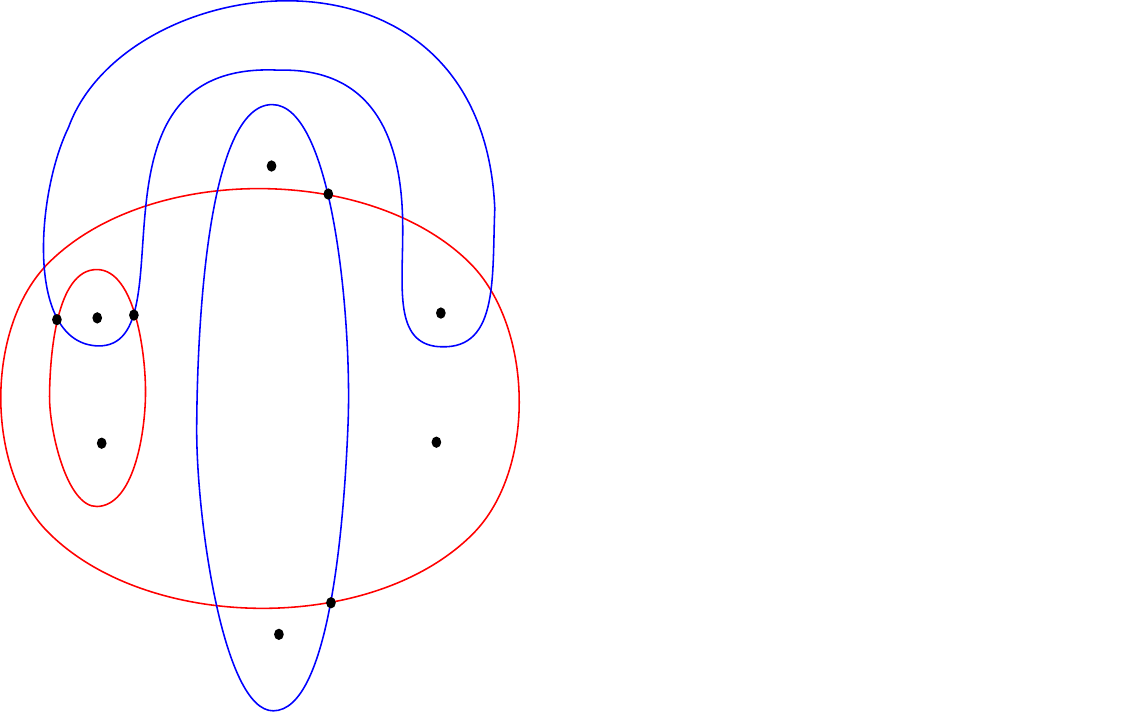
	}
	\]
	\caption{Heegaard diagrams of $L_1'$ and $L_2$}
	\label{fig:link-diagrams}
\end{figure}

We draw the following Heegaard diagrams for $L_1'$ and $L_2$ separately in $(a), (b)$ of Figure \ref{fig:link-diagrams}. Note that there are $16$ generators $(b)$, hence there are no non-trivial differential in this chain complex, and each generator represents a different homology class in the link Floer homology $\widehat{HFL}(L_2)$. For the Heegaard diagram in $(a)$, there are $8$ generators. Among these generators, $4$ of them cancel in pairs, and each of the rest $4$ generators $ T\otimes \theta^{\w}, B\otimes \theta^{\w}, T\otimes \theta^{\w}, B\otimes \xi^{\w} $ generates a different homology class in $\widehat{HFL}(L_1')$, with no different between them.

Recall $A_1$ is the Alexander grading with respect to the basepoints $w_1,z_1$, and $A_2$ is the collapsed Alexander grading with respect to the basepoints $w_2,w_3,z_2,z_3$. Then we can write \[\widehat{HFL}(L_2)\cong \langle a,b,c,d\rangle \otimes \langle x,y,u,v\rangle\] such that the Alexander gradings are as given in Figure \ref{fig:gradings}.

\begin{figure}[h]
	\[
	{
		\fontsize{10pt}{10pt}\selectfont
		\def\svgwidth{5in}
		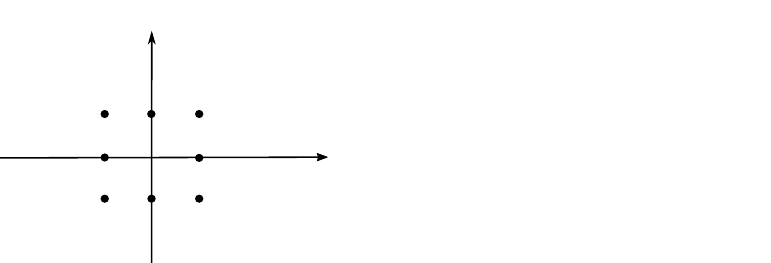
	}
	\]
	\caption{Alexander gradings of $\widehat{HFL}(L_2)$}
	\label{fig:gradings}
\end{figure}

The multi-Alexander grading of a generator is the sum of the corresponding multi-Alexander grading, e.g. \[A_1(d\otimes u) = A_1(d)+A_1(u) = 1/2+1/2=1,\,\,\,A_2(d\otimes u) = A_2(d)+A_2(u) = 1/2-1/2=0.\] 

The element in $\widehat{HFL}(L_2)$ with highest Maslov grading is $d\otimes y$, with $M(d\otimes y)=1,$ which is computed by ignoring the type-$\z$ basepoints and consider it as a multi-pointed Heegaard diagram of $S^3$. The Maslov gradings of other generators could be computed by the relative grading formula of Maslov grading.

The band surgery map $F^{\z}_B$ is defined by counting holomorphic triangles which avoid both type-$\w$ and type-$\z$ basepoints in the triple Heegaard diagram when we draw all the red, blue and green curves in Figure \ref{fig:link-diagrams} together in one diagram. This aligns with Figure $6.2$ in \cite{zemke2019link}, with $A$ being the red curve circling $w_2$ and $z_3$, and $A'$ being the green curve circling $w_3$ and $z_3$. If we ignore the type-$\z$ basepoints in the diagrams, then both diagrams are Heegaard diagrams for $S^3$ with $3$ basepoints. We define \[G^{\z}_B: \widehat{CF}_{a}(S^3,\w) \to \widehat{CF}_{b}(S^3,\w)\] by counting holomorphic triangles which avoid only type-$\w$ basepoints, then $G^{\z}_B$ is a composition of handle-sliding maps, which induces isomorphism on $\widehat{HF}(S^3,\w)$. Here $\widehat{CF}_{a}(S^3,\w)$ and $\widehat{CF}_{b}(S^3,\w)$ refer to the specific Heegaard chain complex generated by diagram $(a)$ and $(b)$ in Figure \ref{fig:link-diagrams}. See \cite[Section 9]{ozsvath2004holomorphic} for details.

Now we focus on the element $G^{\z}_B(B\otimes \xi^{\w})$.  Since we are counting holomorphic triangles of Maslov index $0$ which avoid type-$\w$ basepoints, we have 
\[M(G^{\z}_B(B\otimes \xi^{\w})) = M(B\otimes \xi^{\w})=-2.\] 




Each element in $\widehat{CF}_b(S^3,\w)$ with Maslov grading $-2$ has the collapsed Alexander grading $A= A_1+A_2$ equal to $-1$. (Recall that a term of the form $x_1^{h_1}x_2^{h_2}x_3^{h_3} $ in the normalized Alexander polynomial corresponds to an element in the link Floer homology with collapsed Alexander grading $A = h_1+h_2+h_3$ and Maslov grading $M =A-1$.)  However, following \cite[Lemma 7.2]{zemke2019grading}, we know the change of the multi-Alexander gradings of the band surgery map $F^{\z}_B$, therefore \[A_1(F^{\z}_B(B\otimes \xi^{\w})) = A_1(B\otimes \xi^{\w})=0, \quad \quad A_2(F^{\z}_B(B\otimes \xi^{\w})) = A_2(B\otimes \xi^{\w})+\dfrac{1}{2}=0,\]
so 
\[A(F^{\z}_B(B\otimes \xi^{\w})) = 0.\]

 Now as each term of $G^{\z}_B(B\otimes \xi^{\w})$ has the collapsed Alexander grading different from the expected collapsed Alexander gradings of $F^{\z}_B(B\otimes {\xi^{\w}})$, it means that every holomorphic triangle that contributes to $G^{\z}_B(B\otimes \xi^{\w})$ actually passes through some type-$\z$ basepoints. In other words, there is no contribution from holomorphic triangles that avoid both type-$\w$ and type-$\z$ basepoints to $G^{\z}_B(B\otimes \xi^{\w})$. That is to say, \[F^{\z}_B(B\otimes {\xi^{\w}})=0.\]
Hence, \[F'_P(B) = (F^{\z}_B\circ T^+)(B) = F^{\z}_B(B\otimes \xi^{\w})=0.\]
Then, by Equation \ref{eq:vanishing}, we get \[v\sim 0\] for any $v\in \widehat{HFL}(\mathbb{L},K,\mathbf{k}^+,\mathbf{k}^-)$, which says the cabled link Floer homology $\widehat{cHFL}(\mathbb{L},K)$ vanishes in this case. By Proposition \ref{prop:cabled link homology}, the Floer lasagna module $\mathcal{FL}(W,\mathbb{L})$ vanishes as well. 
	 	 
\end{proof}

\begin{remark}
	If we naively compute the link Floer homology of the link $S^1\times \{pt\} \subset S^1\times S^2$ by the complex represented
\begin{figure}[h]
	\[
	{
		\fontsize{10pt}{10pt}\selectfont
		\def\svgwidth{2.5in}
\begingroup%
  \makeatletter%
  \providecommand\color[2][]{%
    \errmessage{(Inkscape) Color is used for the text in Inkscape, but the package 'color.sty' is not loaded}%
    \renewcommand\color[2][]{}%
  }%
  \providecommand\transparent[1]{%
    \errmessage{(Inkscape) Transparency is used (non-zero) for the text in Inkscape, but the package 'transparent.sty' is not loaded}%
    \renewcommand\transparent[1]{}%
  }%
  \providecommand\rotatebox[2]{#2}%
  \newcommand*\fsize{\dimexpr\f@size pt\relax}%
  \newcommand*\lineheight[1]{\fontsize{\fsize}{#1\fsize}\selectfont}%
  \ifx\svgwidth\undefined%
    \setlength{\unitlength}{277.25583916bp}%
    \ifx\svgscale\undefined%
      \relax%
    \else%
      \setlength{\unitlength}{\unitlength * \real{\svgscale}}%
    \fi%
  \else%
    \setlength{\unitlength}{\svgwidth}%
  \fi%
  \global\let\svgwidth\undefined%
  \global\let\svgscale\undefined%
  \makeatother%
  \begin{picture}(1,0.52251895)%
    \lineheight{1}%
    \setlength\tabcolsep{0pt}%
    \put(0,0){\includegraphics[width=\unitlength,page=1]{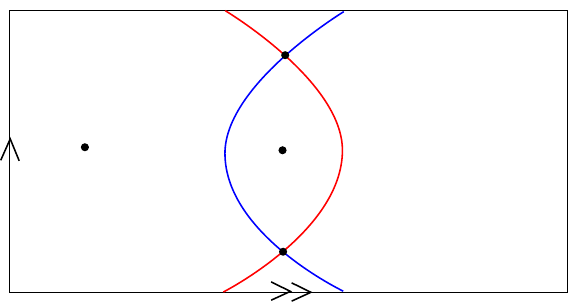}}%
    \put(0.50382447,0.25552528){\makebox(0,0)[lt]{\lineheight{1.25}\smash{\begin{tabular}[t]{l}$w$\end{tabular}}}}%
    \put(0.16882422,0.25552832){\makebox(0,0)[lt]{\lineheight{1.25}\smash{\begin{tabular}[t]{l}$z$\end{tabular}}}}%
    \put(0.51427266,0.41878074){\makebox(0,0)[lt]{\lineheight{1.25}\smash{\begin{tabular}[t]{l}$x$\end{tabular}}}}%
    \put(0.52798619,0.08443365){\makebox(0,0)[lt]{\lineheight{1.25}\smash{\begin{tabular}[t]{l}$y$\end{tabular}}}}%
    \put(0,0){\includegraphics[width=\unitlength,page=2]{s1xs2.pdf}}%
  \end{picture}%
\endgroup%

	}
	\]
	\caption{Heegaard diagram of $S^1\times \{pt\} \subset S^1\times S^2$}
	\label{fig:link-in-s1xs2}
\end{figure}
then there are two generators with a holomorphic bigon connecting them, so we get trivial homology as well. This calculation of the Floer lasagna module $\mathcal{FL}(W,\mathbb{L})=0$ in Proposition \ref{prop:linking 1} indicates that it could work as some generalization of link Floer homology for homologically essential links in general $3$-manifolds.
	
\label{remark:vanishing of hfl}

\end{remark}
 
\begin{remark}
	For $\mathcal{FL}(W,\mathbb{L})$ with $W$ the $4$-manifold obtained by attaching a $2$-handle along the $0$-framed unknot $K$ and general $L$, as all the cobordisms we need to define the cabled link Floer homology $\widehat{cHFL}(\mathbb{L},K)$ could be written as a connected sum of a cobordism near a capping disk of $K$ and the identity cobordism on the remaining part, we expect that using the language of bordered knot Floer homology, one can compute the cabled link Floer homology for general $L$. 
\end{remark}
	 \bibliographystyle{amsalpha}
	 \bibliography{biblio}

\providecommand{\bysame}{\leavevmode\hbox to3em{\hrulefill}\thinspace}
\providecommand{\MR}{\relax\ifhmode\unskip\space\fi MR }
\providecommand{\MRhref}[2]{%
  \href{http://www.ams.org/mathscinet-getitem?mr=#1}{#2}
}
\providecommand{\href}[2]{#2}
\begin{thebibliography}{MWW23}

\bibitem[Dow24]{dowlin2018spectral}
Nathan Dowlin, \emph{A spectral sequence from {K}hovanov homology to knot
  {F}loer homology}, J. Amer. Math. Soc. \textbf{37} (2024), no.~4, 951--1010.
  \MR{4777638}

\bibitem[ETW18]{ehrig2018functoriality}
Michael Ehrig, Daniel Tubbenhauer, and Paul Wedrich, \emph{Functoriality of
  colored link homologies}, Proc. Lond. Math. Soc. (3) \textbf{117} (2018),
  no.~5, 996--1040. \MR{3877770}

\bibitem[Juh16]{juhasz2016cobordisms}
Andr\'as Juh\'asz, \emph{Cobordisms of sutured manifolds and the functoriality
  of link {F}loer homology}, Adv. Math. \textbf{299} (2016), 940--1038.
  \MR{3519484}

\bibitem[JZ20]{juhasz2020contact}
Andr\'as Juh\'asz and Ian Zemke, \emph{Contact handles, duality, and sutured
  {F}loer homology}, Geom. Topol. \textbf{24} (2020), no.~1, 179--307.
  \MR{4080483}

\bibitem[MN22]{manolescu2020skein}
Ciprian Manolescu and Ikshu Neithalath, \emph{Skein lasagna modules for
  2-handlebodies}, J. Reine Angew. Math. \textbf{788} (2022), 37--76.
  \MR{4445546}

\bibitem[MO25]{manolescu2010heegaard}
Ciprian Manolescu and Peter Ozsv\'ath, \emph{Heegaard {F}loer homology and
  integer surgeries on links}, Geom. Topol. \textbf{29} (2025), no.~6,
  2783--3062. \MR{4965183}

\bibitem[MW12]{morrison2010blob}
Scott Morrison and Kevin Walker, \emph{Blob homology}, Geom. Topol. \textbf{16}
  (2012), no.~3, 1481--1607. \MR{2978449}

\bibitem[MWW22]{morrison2019invariants}
Scott Morrison, Kevin Walker, and Paul Wedrich, \emph{Invariants of 4-manifolds
  from {K}hovanov-{R}ozansky link homology}, Geom. Topol. \textbf{26} (2022),
  no.~8, 3367--3420. \MR{4562565}

\bibitem[MWW23]{MWWHandle}
Ciprian Manolescu, Kevin Walker, and Paul Wedrich, \emph{Skein lasagna modules
  and handle decompositions}, Adv. Math. \textbf{425} (2023), Paper No. 109071,
  40. \MR{4589588}

\bibitem[OS04]{ozsvath2004holomorphic}
Peter Ozsv\'ath and Zolt\'an Szab\'o, \emph{Holomorphic disks and topological
  invariants for closed three-manifolds}, Ann. of Math. (2) \textbf{159}
  (2004), no.~3, 1027--1158. \MR{2113019}

\bibitem[OS08]{ozsvath2008holomorphic}
\bysame, \emph{Holomorphic disks, link invariants and the multi-variable
  {A}lexander polynomial}, Algebr. Geom. Topol. \textbf{8} (2008), no.~2,
  615--692. \MR{2443092}

\bibitem[SZ24]{sullivan2024kirbybeltscategorifiedprojectors}
Ian~A. Sullivan and Melissa Zhang, \emph{{K}irby belts, categorified
  projectors, and the skein lasagna module of $s^{2}\times{S^{2}}$}, arXiv
  preprint arXiv:2402.01081 (2024).

\bibitem[Zem17]{zemke2017quasistabilization}
Ian Zemke, \emph{Quasistabilization and basepoint moving maps in link {F}loer
  homology}, Algebr. Geom. Topol. \textbf{17} (2017), no.~6, 3461--3518.
  \MR{3709653}

\bibitem[Zem19a]{zemke2019connected}
\bysame, \emph{Connected sums and involutive knot {F}loer homology}, Proc.
  Lond. Math. Soc. (3) \textbf{119} (2019), no.~1, 214--265. \MR{3957835}

\bibitem[Zem19b]{zemke2019grading}
\bysame, \emph{Link cobordisms and absolute gradings on link {F}loer homology},
  Quantum Topol. \textbf{10} (2019), no.~2, 207--323. \MR{3950650}

\bibitem[Zem19c]{zemke2019link}
\bysame, \emph{Link cobordisms and functoriality in link {F}loer homology}, J.
  Topol. \textbf{12} (2019), no.~1, 94--220. \MR{3905679}

\end{thebibliography}
\end{document}